\documentclass[11pt,a4j,twoside]{article}

\usepackage{amsmath,amssymb}
\usepackage{amsthm}
\usepackage[abbrev]{amsrefs}
\usepackage{latexsym}
\usepackage{graphicx}

\usepackage[english]{babel}
\usepackage{fancyhdr}
\usepackage{a4wide}

\usepackage{bm}

\allowdisplaybreaks[1]

\numberwithin{equation}{section}

\newtheorem{thm}{Theorem}[section]
\newtheorem{prop}[thm]{Proposition}
\newtheorem{lem}[thm]{Lemma}
\newtheorem{cor}[thm]{Corollary}
\newtheorem{dfn}[thm]{Definition}
\theoremstyle{definition} 
\newtheorem{ex}[thm]{Example}
\newtheorem{rem}[thm]{Remark}

\newcommand\ND{\newcommand}

\ND\lref[1]{Lemma~\ref{#1}}
\ND\tref[1]{Theorem~\ref{#1}}
\ND\pref[1]{Proposition~\ref{#1}}
\ND\sref[1]{Section~\ref{#1}}
\ND\ssref[1]{Subsection~\ref{#1}}
\ND\aref[1]{Appendix~\ref{#1}}
\ND\rref[1]{Remark~\ref{#1}}
\ND\cref[1]{Corollary~\ref{#1}}
\ND\eref[1]{Example~\ref{#1}}
\ND\fref[1]{Fig.\ {#1} }
\ND\lsref[1]{Lemmas~\ref{#1}}
\ND\tsref[1]{Theorems~\ref{#1}}
\ND\dref[1]{Definition~\ref{#1}}
\ND\psref[1]{Propositions~\ref{#1}}
\ND\rsref[1]{Remarks~\ref{#1}}
\ND\sssref[1]{Subsections~\ref{#1}}
\ND\esref[1]{Examples~\ref{#1}}

\newcommand{\ep}{\varepsilon}
\newcommand{\h}{\quad}

\title{Convergence of martingales with jumps on submanifolds of Euclidean spaces and its applications to harmonic maps}
\date{}
\author{Fumiya Okazaki}
\begin{document}
\maketitle
\footnote{Mathematical Institute, Graduate School of Science, Tohoku University, Sendai, Japan}
\footnote{MSC2020 Subject Classifications: 60G44, 60J45}
\footnote{Email address: fumiya.okazaki.q4@dc.tohoku.ac.jp}
\renewcommand{\thepage}{\arabic{page}}
\begin{abstract}
Martingales with jumps on Riemannian manifolds and harmonic maps with respect to Markov processes are discussed in this paper. Discontinuous martingales on manifolds were introduced in Picard (1991). We obtain results about the convergence of martingales with finite quadratic variations on Riemannian submanifolds of higher dimensional Euclidean space as $t\to \infty$ and $t\to 0$. Furthermore we apply the result about martingales with jumps on submanifolds to harmonic maps with respect to Markov processes such as fractional harmonic maps.
\end{abstract}
\section{Introduction}
Martingales on manifolds have been studied in connection with harmonic maps. It is well known that given a harmonic map between Riemannian manifolds, we can obtain a martingale on the target manifold by inserting a Brownian motion on the domain manifold into the map. In \cite{Ken} and \cite{Atsu}, some Liouville-type theorems for harmonic maps were shown by probabilistic methods. To prove some of the Liouville-type theorems, they used the convergence theorem for martingales on a Riemannian manifold shown in \cite{Dar2}. The convergence of backward continuous martingales as $t\to 0$ was also considered in \cite{JBWalsh}, \cite{Sharpe}, \cite{CG}, \cite{Eme2} and \cite{HeZheng}. The convergence of backward martingales is related to singularities of harmonic functions and harmonic maps.\\

The aim of our research is to extend the above results to discontinuous martingales on Riemannian submanifolds of higher dimensional Euclidean spaces. First we describe our setting and main result. Except for \sref{harmonic}, we always assume that we are given a filtered probability space $(\Omega, \mathcal{F}, \{ \mathcal{F}_t\}_{t\geq 0}, \mathbb{P})$ and the usual hypotheses for $\{ \mathcal{F}_t \}_{t\geq 0}$ hold. A stochastic process $X$ valued in a manifold $M$ is called an $M$-valued semimartingale if $f(X)$ is an $\mathbb{R}$-valued semimartingale for all $f\in C^{\infty}(M)$. In \cite{Pic1}, a map $\gamma$ from $M\times M$ to $TM$ called a connection rule was introduced to define the It\^o integral on a manifold. We review connection rules in \sref{pre}. Given a connection rule $\gamma$, we can determine directions of jumps of a semimartingale $X$ by $\Delta X_s = \gamma (X_{s-}, X_s)$. Furthermore we can define the It\^o integral of 1-form along a discontinuous semimartingale by a connection rule $\gamma$ by \cite{Pic1}. It is denoted by $\displaystyle \int \alpha(X_-)\, \gamma dX$, where $\alpha$ is a 1-form. A semimartingale $X$ is called a $\gamma$-martingale if for any 1-form, the It\^o integral along $X$ with respect to $\gamma$ is an $\mathbb{R}$-valued local martingale. Note that the definition of discontinuous martingale on a manifold depends on the directions of jumps $\Delta X$ because the stochastic integral along discontinuous semimartingales can be defined only when directions of jumps are determined. As shown in \cite{Pic1}, each connection rule determines a torsion-free connection on $M$. Thus we need somewhat more information than that of connections to define the It\^o integral. If $X$ is continuous, we can define the It\^o integral along $X$ with respect to a connection $\nabla$. The basic properties of manifold-valued continuous martingales can be found in \cite{Dar1}, \cite{Eme}, and \cite{Hsu}. Given a Riemannian metric $g$ on a manifold $M$, we can consider continuous martingales on $M$ with respect to the Levi-Civita connection. In \cite{Dar2}, Darling and Zheng have proven that an $M$-valued continuous martingale $X$ with respect to the Levi-Civita connection converges as $t\to \infty$ in the one-point compactification of $M$ almost surely on $\{ \int_0^{\infty}g(X_s)\, d[X,X]_s<\infty \}$; see also \cite{Mey}. In this sense, one of our main results, \tref{conv2} below, can be seen as a convergence theorem for discontinuous martingales on submanifolds of higher dimensional Euclidean spaces. The main idea of the proof is based on \cite{Dar2}. Let $M$ be an $n$-dimensional Riemannian submanifold of $\mathbb{R}^d$. We define the connection rule $\eta$ on $M$ by
\begin{gather}\label{eta}
\eta (x,y)=\Pi_x(y-x),\ x,y\in M,
\end{gather}
where $\Pi_x:\mathbb{R}^n\to T_xM$ is the orthonormal projection. Then for a semimartingale $X$ on $M$, we can define the direction of jumps by
\begin{gather*}
\Delta X_s=\eta (X_{s-},X_s),\ s\geq 0.
\end{gather*}
\begin{thm}\label{conv2}
Let $M$ be a complete Riemannian submanifold of $\mathbb{R}^d$ and $X$ an $M$-valued $\eta$-martingale with
\[
\mathbb{E}\left[ [X,X]_{\infty} \right]<\infty,
\]
where $[X,X]$ is the quadratic variation of $X$ as an $\mathbb{R}^d$-valued semimartingale. Then $X_t$ converges in $M_{\Delta}$ as $t\to \infty$ with probability one, where $M_{\Delta}$ is the one-point compactification of $M$. Furthermore if $|\Delta X|$ is uniformly bounded on $[0,\infty)\times \Omega$, then the convergence of $X$ occurs almost surely on $\{ [X,X]_{\infty}<\infty \}$. 
\end{thm}
Compared with \cite{Dar2}, we use the quadratic variation as an $\mathbb{R}^d$-valued semimartingale instead of the Riemannian quadratic variation $\int g(X)\, d[X,X]$. If $X$ is continuous, the two quadratic variations coincide. However if $X$ has jumps, they are different and we can easily construct a martingale which does not converge even though the latter quadratic variation equals zero. See \rref{counterexample}.

Next we will show a convergence theorem of $M$-valued martingales on $(0,\infty)$ as $t\to 0$ as considered in \cite{Eme2}. A martingale with a parameter $t>0$ was first considered in \cite{JBWalsh} concerning singularities of complex functions and the result in \cite{JBWalsh} was extended to general continuous local martingales in \cite{Sharpe} and \cite{CG}. Such results for manifold-valued continuous martingale were shown in \cite{Eme2} and \cite{HeZheng}. The method of \cite{HeZheng} was the modification of \cite{HeYanZheng}. Our next statement is the discontinuous version of \'Emery's result. We emphasize that some facts which hold only for continuous martingales were used in the proofs in \cite{Eme2} and \cite{HeZheng}. In particular the fact that the composition of a manifold-valued martingale and a convex function with respect to geodesics on the manifold is a submartingale does not hold for manifold-valued discontinuous martingales in our cases, e.g. Picard \cite{Pic2}. In \sref{main}, we will introduce an $M$-valued martingale with an end point in $\mathbb{R}^d$. It is an $M$-valued martingale which allows inside killing.
\begin{thm}\label{convtto0}
Let $M$ be a compact submanifold of $\mathbb{R}^d$ and $(\{X_t\}_{t>0},\zeta,p)$ an $M$-valued $\eta$-martingale with a killing time $\zeta$ and an end point $p$ indexed by $t \in (0,\infty)$.
\begin{itemize}
\item[(1)]
Suppose $\displaystyle X_0:=\lim_{t\to 0}X_t$ exists almost surely. Then $(\{X_t\}_{t\geq 0},\zeta,p)$ is an $M$-valued $\eta$-martingale with the end point $p$.
\item[(2)]Suppose $\displaystyle \lim_{\varepsilon \to 0}[X,X]^{\varepsilon}_t<\infty$ almost surely. Then $\displaystyle \lim_{t\to 0}X_t$ exists almost surely, where
\[
[X,X]^{\varepsilon}_t:=\int_{(\varepsilon,t]}d[X,X]_s.
\]
\end{itemize}
\end{thm}
We apply the theory of martingales with jumps on a submanifold to harmonic maps with respect to non-local Dirichlet forms. Fractional harmonic maps introduced in \cite{Lio, Lio2} are one of the most typical examples of such harmonic maps. A fractional harmonic map is a harmonic map with respect to the fractional Laplacian $(-\Delta)^{\frac{\alpha}{2}}$, $\alpha \in (0,2)$. The energy of a map $u$ from $\mathbb{R}^m$ to a compact submanifold $M$ of $\mathbb{R}^d$ with respect to the fractional Laplacian is defined by
\[
\mathcal{E}(u)=\int_{\mathbb{R}^m}|(-\Delta)^{\frac{\alpha}{4}} u(z)|^2dz,
\]
and the harmonic map is its stationary map. In \sref{harmonic}, we will see that we can obtain a discontinuous martingale on a submanifold by inserting an $\alpha$-symmetric stable process into a map satisfying the Lagrange equation with respect to the fractional Laplacian in some situations. In this way, it can be expected that the fractional harmonic map discussed in \cite{Lio, Lio2} can be studied by using stochastic processes. Applying \tref{conv2}, we also show the convergence of discontinuous martingales on manifolds valued in a small domain in a circle.
We also show a Liouville-type theorem for fractional harmonic maps valued in a circle without any conditions for domain spaces and Markov processes except for the Liouville property for bounded harmonic functions. Furthermore we also give a sufficient and necessary condition for the fine continuity of harmonic maps from the probabilistic point of view as an application of \tref{convtto0}.\\

We give an outline of the paper. In \sref{pre}, we review connection rules introduced in \cite{Pic1} and use them to define the stochastic integral of 1-form and the quadratic variation of 2-tensor along discontinuous semimartingales on manifolds. In \sref{main}, we give the proofs of Theorems \ref{conv2} and \ref{convtto0}. Then we apply our main theorems to harmonic maps for non-local operators in \sref{harmonic}.\\

Throughout this paper, given a metric space $M$, we denote by $C_0(M)$ the set of all continuous functions on $M$ with compact support. In the case that $M$ is a manifold, we denote by $C_0^{\infty}(M)$ the set of all $C^{\infty}$ functions with compact support. For two manifolds $M$ and $N$ and a positive integer $k$, we denote by $C^k(M\, ; N)$ the set of all $C^k$ maps from $M$ to $N$.
\section{Stochastic integral on manifolds}\label{pre}
In this section, we review connection rules and stochastic integrals along c\`adl\`ag semimartingales on manifolds which were introduced in \cite{Pic1}. The stochastic integrals along c\`adl\`ag semimartingales were further studied in \cite{Co1, Co2} and the definition of connection rules was extended in \cite{Co2}. Several works done in \cite{Co1, Co2} were summarized in \cite{Mai}. In this article, we review the definition of connection rules considered in \cite{Co2}. We suppose that $M$ is a paracompact $C^{\infty}$ manifold. We set
\[
\mathrm{diag}(M):=\{ (x,y)\in M\times M \mid x=y \}.
\]
\begin{dfn}
Let $\gamma:M\times M\to TM$ be a measurable map and suppose $\gamma$ is $C^2$ on a neighborhood of $\mathrm{diag}(M)$, that is, there exists an open neighborhood $\mathcal{U} \subset M \times M$ of $\mathrm{diag}(M)$ such that $\gamma \in C^2(\mathcal{U}\, ;TM)$. Then $\gamma$ is called a \textbf{connection rule} if it satisfies the following conditions, for all $x,y\in M$,
\begin{enumerate}
\item[(i)] $\gamma (x,y)\in T_xM$;
\item[(ii)] $\gamma (x,x)=0$;
\item[(iii)] $(d \gamma (x,\cdot))_x=id_{T_xM}.$
\end{enumerate}
\end{dfn}
\begin{ex}\label{conne1}
If $M=\mathbb{R}^d$, the map $\gamma$ defined by
\[
\gamma (x,y)=y-x,\ x,y\in M
\]
is a connection rule.
\end{ex}
\begin{ex}\label{conne2}
Let $M$ be a submanifold of $\mathbb{R}^N$ and $\Pi _x:\mathbb{R}^N\to T_xM$ an orthonormal projection for each $x\in M$. Then
\[
\gamma (x,y)=\Pi_x (y-x),\ x,y\in M
\]
is a connection rule.
\end{ex}
\begin{ex}\label{conne3}
Let $M$ be a strongly convex Riemannian manifold. Then
\begin{gather}
\gamma (x,y)=\exp_x^{-1}y,\ x,y\in M \label{conne3def}
\end{gather}
is a connection rule.
\end{ex}
As mentioned in \cite{Pic1}, a connection rule induces a torsion-free connection. The connection induced by a connection rule given in \esref{conne1} through \ref{conne3} is the Levi-Civita connection.
\begin{dfn}\label{Pic3.1}
An $M$-valued stochastic process $X$ is called an \textbf{$\bm{M}$-valued semimartingale} if for all $f\in C^{\infty}(M)$, $f(X)$ is an $\mathbb{R}$-valued c\`{a}dl\`{a}g semimartingale.
\end{dfn}
The following \pref{Pic3.2} shown in \cite{Pic1} determines the It\^o integral of 1-form and the quadratic variation of 2-tensor along c\`{a}dl\`{a}g semimartingales on manifolds.
\begin{prop}[\cite{Pic1}, Proposition 3.2]\label{Pic3.2}
Let $\gamma$ be a connection rule, $X$ an $M$-valued semimartingale, and $\phi$ a $T^*M$-valued process above $X$, that is, $\phi_t\in T^*_{X_t}M$ for all $t\geq 0$ almost surely. Let $\{ U_i\}_{i=1}^{\infty}$ be an atlas such that $\gamma$ is differentiable on each $U_i\times U_i$ and each $U_i$ appears an infinite number of times in the sequence. Define the sequence of stopping times $\{ \sigma_i \}_{i=0}^{\infty}$ by $\sigma_0=0$ and
\[
\sigma_i=\inf \{ s\geq \sigma_{i-1}\mid X_s \notin U_i \}.
\]
for $i\geq 1$. Let
\[
\Delta^n :0=\tau_0^n<\tau_1^n<\dots <\tau_{k_n}^n
\] 
be a random partition tending to infinity, i.e.
\[
\lim _{n\to \infty} |\Delta ^n | = 0,
\]
where $|\Delta ^n | =\max \{ |\tau_i^n-\tau_{i-1}^n|\, ;\,i=1,\dots ,k_n\}$ and
\[
\lim_{n\to \infty} \tau_{k_n}^n=\infty.
\]
Moreover, suppose that $\Delta ^n$ contains $\{ \sigma_i \land \tau_{k_n}^n \}_{i=1}^{\infty}$ for all $n$.
Set
\[
J_t^n:=\sum _{i=1}^{k_n} \langle \phi _{\tau_{i-1}^n\land t}, \gamma (X_{\tau_{i-1}^n\land t},X_{\tau_i^n\land t})\rangle,
\]
where $s\land t=\min \{s,t\}$ for $s,t\geq 0$. Then $J_t^n$ converges in probability as $n\to \infty$ for every $t\geq 0$ and the limit $J_t$ is independent of the partition. Furthermore, the process $\{ J_t \}_{t\geq 0}$ has a modification which is a c\`{a}dl\`{a}g semimartingale.
\end{prop}
We denote $\{ J_t\}_{t\geq 0}$ by $\displaystyle \int \phi_-\, \gamma dX$. This can be considered as the \textbf{stochastic integral along $\bm{X}$ with respect to $\bm{\gamma}$}. We can construct the \textbf{quadratic variation of 2-tensor} with respect to a connection rule in a similar way. Let $b$ be a $T^*M\otimes T^*M$-valued c\`{a}dl\`{a}g adapted process above $X$. Let 
\[
\Delta^n :0=\tau_0^n<\tau_1^n<\dots <\tau_{k_n}^n,\ n=1,2,\dots
\]
be a sequence of random partitions tending to infinity as in \pref{Pic3.2}. Then we can define the quadratic variation of $b$ with respect to a connection rule $\gamma$ by the limit as $n\to \infty$ in probability of
\[
\sum_{k=0}^{k_n-1}b_{\tau_k^{n}\land t}(\gamma(X_{\tau^n_k\land t},X_{\tau^n_{k+1}\land t}),\gamma(X_{\tau^n_k \land t},X_{\tau^n_{k+1}\land t}))
\]
and denote this by $\int_0^tb_{s-}\, \gamma d[X,X]_s$. Furthermore, set
\begin{align*}
\int_0^tb_{s-}\, \gamma d[X,X]^d_s&:=\sum_{0<s\leq t}b_{s-}(\gamma (X_{s-},X_s),\gamma(X_{s-},X_s)),\\
\int_0^tb_{s-}\, d[X,X]^c_s&:=\int_0^tb_{s-}\, \gamma d[X,X]_s-\int_0^tb_{s-}\, \gamma d[X,X]^d_s.
\end{align*}
Then $\displaystyle \int b_-\ d[X,X]^c$ is continuous and does not depend on  the choice of a connection rule. See Proposition 3.6 of \cite{Pic1} for details.
\begin{prop}[\cite{Pic1}, Proposition 3.7]\label{Ito1}
Let $X$ be an $M$-valued semimartingale. Then for each connection rule $\gamma$ and $f\in C^2(M)$, it holds that
\begin{align}
f(X_t)=f(X_0)+\int_0^tdf(X_{s-})\, \gamma dX_s+\frac{1}{2}\int_0^t \nabla df(X_{s-})\, d[X,X]^c_s\nonumber \\
+\sum_{0<s\leq t} \{ f(X_s)-f(X_{s-})-\langle df(X_{s-}), \gamma (X_{s-}, X_s)\rangle \}.\label{Itosformula}
\end{align}
\end{prop}
For a semimartingale $X$, a connection rule $\gamma$ and $f\in C^2(M)$, we use the following notation:
\[
N^f(X)=N^f:=\int df(X_-)\, \gamma dX.
\]
\section{Proofs of the main theorems}\label{main}
\subsection{Proof of \tref{conv2}}
In this section, we start with a review of the definition of $\gamma$-martingales with jumps introduced in \cite{Pic1}. Then we prove the convergence of martingales with jumps on manifolds. The main idea of the proof of our theorem is based on \cite{Dar2}.
\begin{dfn}[\cite{Pic1}, Definition 4.1]
Let $M$ be a manifold with a connection rule $\gamma$ and $X$ an $M$-valued semimartingale. We call $X$ a \textbf{$\bm{\gamma}$-martingale} if for all $T^*M$-valued c\`adl\`ag processes $\alpha _s$ above $X_-$, $\displaystyle \int \alpha _{s-}\, \gamma dX_s$ is a local martingale.
\end{dfn}
In \cite{Pic1}, it is mentioned that a semimartingale $X$ is a $\gamma$-martingale if and only if $N^f(X)$ is a local martingale for all $f\in C^{\infty}(M)$. Furthermore we only have to check this for a finite subset of $C^{\infty}(M)$ by the following lemma.
\begin{lem}\label{martingalelem1}
Take a Riemannian metric $g$ on $M$ and an isometric immersion $\iota:M\to \mathbb{R}^d$. Let $\gamma$ be a connection rule on $M$ which induces the Levi-Civita connection with respect to $g$. Then the following are equivalent.
\begin{itemize}
\item[(a)]For all $f\in C^{\infty}(M)$, $N^f(X)$ is a local martingale.
\item[(b)]For each $i=1,\dots,d$, $N^{\iota^i}(X)$ is a local martingale.
\item[(c)]$X$ is a $\gamma$-martingale.
\end{itemize}
\end{lem}
\begin{proof}
(a)$\Rightarrow$(b) is obvious and (c)$\Rightarrow$(a) is obtained by It\^o's formula in \pref{Ito1}. We will show (b)$\Rightarrow$(c). Take any local coordinate $(x^1,\dots,x^n)$ on $M$. Then it holds that
\[
\frac{\partial}{\partial x^k}=\sum_{i=1}^n\frac{\partial \iota^i}{\partial x^k}\nabla \iota^i.
\]
Therefore for all $T^*M$-valued c\`adl\`ag processes $\alpha _s$, we obtain
\[
\int_0^t\alpha_{s-}\, \gamma dX_s=\sum_{i=1}^n\int_0^t\langle \alpha_{s-},\nabla \iota^i(X_{s-})\rangle\, dN^{\iota^i}(X)_s
\]
and the right-hand side is a local martingale.
\end{proof}
In this section, we always assume that $M$ is an $n$-dimensional complete Riemannian manifold and $\nabla$ is the Levi-Civita connection unless otherwise stated. Besides, we only consider the connection rule $\eta$ in the form of \eqref{eta}. To show the convergence theorem, we prove \lref{lem2} below needed later on. \lref{lem2} and \cref{cor2} below were essentially shown in Lemma 2 of \cite{Dar2} in the same way but we do not assume the continuity of the process $X$ in \cref{cor2}. Moreover, to prove Lemma 2 of \cite{Dar2}, actually we had to show that it sufficed to take a function $f$ from a countable subset of $C^{\infty}_0(M)$ though the process was omitted. Thus we confirm the proofs of \lref{lem2} and \cref{cor2} in this section. Since $M$ is separable, we can take a countable set $N\subset M$ satisfying $\overline{N}=M$. Let
\[
\mathbb{Q}_1=\{q\in \mathbb{Q}\mid \ 0<2q<\mathrm{diam}(M)\leq \infty \},
\]
where $\mathrm{diam}(M)$ is the diameter of $M$. For each $x \in N$ and $q\in \mathbb{Q}_1$, we choose $\phi_{q,x}\in C_0^{\infty}(M)$ satisfying
\[
\phi_{q,x}(y)=\left\{ \begin{array}{ll}
1 & (y\in B_q(x)), \\
0 & (y\in B_{2q}(x)^c),
  \end{array} \right.
\]
and
\[
0< \phi_{q,x}(y) < 1\ (y\in B_{2q}(x) \backslash B_q(x) ),
\]
where $B_q(x)$ is a geodesic ball of radius $q$ centered at $x$. Let $\Phi=\{\phi_{q,x}\mid q\in \mathbb{Q}_1, x\in N\}$. Then $\Phi$ is a countable subset of $C_0^{\infty}(M)$. Let $M_{\Delta}$ be the one-point compactification of $M$.
\begin{lem}\label{lem2}
Let $c:[0,\infty )  \to M$ be a map. Suppose for all $\phi \in \Phi$, $\displaystyle \lim_{t \to \infty} \phi(c(t))$ exists in $\mathbb{R}$. Then $\displaystyle \lim_{t \to \infty} c(t)$ exists in $M_{\Delta}$.
\end{lem}
\begin{proof}
Suppose for all $\phi _{q,x}\in \Phi$,  $\displaystyle \lim_{t \to \infty} \phi_{q,x}(c(t))=0$. Then for every compact set $K$,  $c(t)$ lies outside of $K$ for all sufficiently large $t$. Therefore $\displaystyle \lim_{t \to \infty} c(t)$ is the point at infinity.\\
\indent On the other hand, if $\displaystyle \lim_{t \to \infty} \phi_{p,x}(c(t))\neq 0$ for some $\phi_{p,x}\in \Phi$, then $c(t)\in B_{2p}(x)$ for all sufficiently large $t$. Because $\overline{B_{2p}(x)}$ is compact, we can choose a sequence $\{t_k\}$ such that $t_k\to \infty$ as $k\to \infty$ and the sequence $\{c(t_k) \}$ converges in $B_{2p}(x)$. Let $\displaystyle y=\lim_{k\to \infty}c(t_k)$. For $q\in \mathbb{Q}_1$, we can choose $z\in N$ satisfying $\displaystyle d(z,y)<\frac{q}{4}$. By assumption, $\displaystyle \lim_{t\to \infty} \phi_{\frac{q}{2},z}(c(t))$ exists. On the other hand, $\displaystyle \lim_{k\to \infty} c(t_k)=y\in B_{\frac{q}{2}}(z)$. Hence
\[
\lim_{k\to \infty}\phi_{\frac{q}{2},z}(c(t_k))=\phi_{\frac{q}{2},z}(y)=1.
\]
Therefore
\[
\lim_{t\to \infty}\phi_{\frac{q}{2},z}(c(t))=1
\]
because it does not depend on the way to choose a sequence $\{t_k\}$.
Therefore $\displaystyle c(t)\in B_{\frac{q}{2}}(z)\subset B_q(y)$ for all sufficiently large $t$. Since $q\in \mathbb{Q}_1$ is arbitrary, we have $\displaystyle \lim_{t \to \infty} c(t)=y$.
\end{proof}
\begin{cor}\label{cor2}
Let $X$ be an $M$-valued process. Suppose for all $f\in C^{\infty}_0(M)$,
\[
\mathbb{P}\left( \lim_{t\to \infty} f(X_t)\ \text{exists in}\ \mathbb{R}\right) =1.
\]
Then $\displaystyle \lim_{t \to \infty} X_t$ exists in $M_{\Delta}$ almost surely.
\end{cor}
\begin{proof}
By assumption, $\displaystyle \lim_{t\to \infty}\phi_{q,x}(X_t)$ exists almost surely for any $\phi_{q,x}\in \Phi$. Let $\Gamma_{q,x}$ be the exceptional set and $\displaystyle \Gamma =\bigcup_{q,x} \Gamma_{q,x}$. Then $\Gamma$ is a null set and for any $\omega \in \Omega \, \backslash \, \Gamma$ and $\phi_{q,x} \in \Phi$, $\displaystyle \lim_{t\to \infty}\phi_{q,x}(X_t(\omega))$ exists. Therefore $\displaystyle \lim_{t\to \infty} \phi_{q,x}(X_t)$ exists on $\Omega \, \backslash \,  \Gamma$ and this proves the corollary by \lref{lem2}.
\end{proof}

\begin{proof}[Proof of \tref{conv2}]
For $f\in C_0^{\infty}(M)$, we can construct a function $\bar f \in C_0(\mathbb{R}^n)$ satisfying
\begin{gather}
f(x)=\bar f(x),\ \nabla f(x)=D \bar f(x)\label{nabla},\ x\in M,
\end{gather}
where $D \bar f$ is the gradient of $\bar f$ on $\mathbb{R}^n$. Indeed, let $T^{\perp}M$ be a normal bundle of $M$ and $\exp^{\perp}$ a normal exponential map. Then it is well known that there exists a neighborhood $\mathcal{U}$ of the zero section of $T^{\perp}M$ such that $\exp^{\perp}:\mathcal{U}\to \exp^{\perp}(\mathcal{U})$ is a diffeomorphism. Set $\phi :=f\circ \pi_{T^{\perp}M}$, where $\pi_{T^{\perp}M}:T^{\perp}M\to M$ is a projection and $\bar f:=\phi \circ (\exp^{\perp})^{-1}$. Then $\bar f$ is a differentiable function on $\exp^{\perp}(\mathcal{U})$ and \eqref{nabla} holds. By extending $\bar f$ to the outside of $\exp^{\perp}(\mathcal{U})$ as a differential function on $\mathbb{R}^n$, we obtain a desired extension of $f$. Let $\iota^1,\dots,\iota^n$ be the coordinate functions, that is, $\iota(x)=(\iota^1(x),\dots,\iota^n(x))\in \mathbb{R}^n$ for $x\in M$. Then $[X,X]=[\iota(X),\iota(X)]$. Let
\begin{align*}
N_t &= \int_{0+}^t\partial_i \bar f(X_{s-})\, d\iota^i(X)_s,\\
A_t &= \frac{1}{2}\int_{0+}^t \partial^2_{ij}\bar f(X_{s-})\, d[\iota^i(X)^c,\iota^j(X)^c]_s,\\
B_t &= \sum_{0<s\leq t}\{ \bar f(X_s) -\bar f(X_{s-})-\langle D\bar f(X_{s-}),\Delta \iota (X)_s\rangle_{\mathbb{R}^n} \},
\end{align*}
where $\iota^i(X)^c$ is a continuous local martingale part of $\iota^i(X)$. Then by It\^o's formula for semimartingales on $\mathbb{R}^n$, it holds that
\begin{align*}
f(X_t)-f(X_0)&=\bar f(X_t)-\bar f(X_0)\\
&=N_t+A_t+B_t.
\end{align*}
First, we will show the convergence of $N_t$. For stopping times $\sigma,\tau$ with $\sigma \leq \tau$, it holds that
\begin{align*}
\langle D \bar f(X_{\sigma}),\iota (X_{\tau})-\iota(X_{\sigma})\rangle_{\mathbb{R}^n}=\langle df(X_{\sigma}),\eta (X_{\sigma},X_{\tau})\rangle
\end{align*}
by \eqref{nabla}. Therefore
\begin{gather*}
N_t=N^f_t
\end{gather*}
and consequently, it is a local martingale. Since the support of $\bar{f}$ is compact, there exists $K\geq 0$ such that for all $x\in \mathbb{R}^d$,
\[
\| D\bar f(x)\| \leq K,\ \| \mathrm{Hess}(\bar f)(x)\| \leq K.
\]
Thus it holds that
\begin{align*}
[N,N]_t&=\int_0^t\partial_i \bar f(X_{s-})\, \partial_j \bar f(X_{s-})\, d[\iota^i(X),\iota^j(X)]_s\\
&\leq K^2[\iota (X),\iota(X)]_t,\\
\| A\|_t& \leq \frac{K}{2}[\iota(X),\iota(X)]^c_t.
\end{align*}
In a similar way, we obtain 
\begin{gather*}
\|B\|_t\leq \frac{K}{2}\sum _{0<s\leq t}|\Delta \iota (X)_s|^2.
\end{gather*}
Hence $N_t$, $A_t$ and $B_t$ converge almost surely as $t\to \infty$. Therefore $f(X)$ also converges almost surely and consequently, $X_{\infty}$ exists in $M_{\Delta}$ almost surely by \cref{cor2}.
\end{proof}

\begin{rem}\label{counterexample}
If we replace $[X,X]$ by the Riemannian quadratic variation $\int g(X_-)\, \gamma d[X,X]$ in the assumption for the quadratic variation, \tref{conv2} does not hold in general. In fact, we can construct martingales with integrable Riemannian quadratic variations on
\[
S^1=\{(x^1,x^2)\in \mathbb{R}^2\mid (x^1)^2+(x^2)^2=1\}
\]
which do not converge. Let $N_t$ be a Poisson process with $\mathbb{E}\left[ N_1 \right] =\lambda >0$. Let
\[
X_t=e ^{i \pi N_t}.
\]
Then $X$ is a $\gamma$-martingale since $\gamma(X_-,X)=0$. Thus the Riemannian quadratic variation of $X$ is $\int g(X_-)\, \gamma d[X,X]=0$. However obviously $X$ does not converge as $t\to \infty$. 
\end{rem}

\subsection{Proof of \tref{convtto0}}
First we define martingales with end points.
\begin{dfn}
Let $X$ be an $\mathbb{R}^d$-valued semimartingale, $\zeta$ an $\{\mathcal{F}_t\}_{t\geq 0}$-stopping time and $p$ a point in $\mathbb{R}^d$. We call $(X,\zeta, p)$ an \textbf{$\bm{M}$-valued semimartingale with the end point $\bm{p}$} if it satisfies $X_t\in M$ for $t\in [0,\zeta)$ and $X_t=p$ for $t \geq \zeta$ a.s.
\end{dfn}
If $p\in M$ or $\zeta=\infty$ almost surely, then an $M$-valued semimartingale $X$ with the end point $p$ is just an $M$-valued semimartingale.
\begin{rem}
An $M$-valued semimartingale $X$ with an end point $p \in \mathbb{R}^d$ is not an $M$-valued process in a strict sense, but we can define the stochastic integral for each vector field $V$ on $M$ along $\{X_t\}_{t\geq 0}$. Indeed, $V$ can be extended to a map $\bar{V}\colon \mathbb{R}^d\to \mathbb{R}^d$ and we can define the stochastic integral $\displaystyle \int \langle \bar{V}(X_-), dX\rangle$. This is independent of the extension since for $t\in [0,\zeta )$, $X_t \in M$ and for $t\geq \zeta$, $X_t=p$. Thus we denote the integral simply by $\displaystyle \int \langle V(X_-), dX\rangle$.
\end{rem}
\begin{dfn}
Let $(X,\zeta, p)$ be an $M$-valued $\{\mathcal{F}_t\}_{t\geq 0}$-semimartingale with the end point $p$. The triple $(X,\zeta, p)$ is called an \textbf{$\bm{M}$-valued $\bm{\eta}$-martingale with the end point $\bm{p}$} if for any bounded vector field $V$ on $M$, the stochastic integral $\displaystyle \int \langle V(X_-), dX \rangle$ is a local martingale.
\end{dfn}
\begin{rem}
The same argument as \lref{martingalelem1} still holds for a martingale with an end point. We can prove this by replacing $\iota$ by $\bar{\iota}$ which is an extension of $\iota$ satisfying \eqref{nabla}.
\end{rem}
\begin{dfn}
Let $\{X_t\}_{t>0}$ be a process indexed by $t\in (0,\infty)$. Suppose that there exist $p\in \mathbb{R}^d$ and an $\{\mathcal{F}_t\}_{t\geq 0}$-stopping time $\zeta$ such that $X_t \in M$ for $t\in [0,\zeta)$ and $X_t=p$ for $t\geq \zeta$. We call $(\{X_t\}_{t>0},\zeta,p)$ an \textbf{$\bm{M}$-valued $\bm{\eta}$-martingale with an end point indexed by $\bm{t\in (0,\infty)}$} if for all $\varepsilon>0$, $(\{X_{t+\ep}\}_{t\geq 0},(\zeta-\ep)\lor 0, p)$ is an $M$-valued $\{\mathcal{F}_{t+\ep} \}_{t\geq 0}$-martingale with the end point $p$, where $s\lor t=\max \{s,t \}$ for $s,t\geq 0$.
\end{dfn}
For an $\mathbb{R}$-valued semimartingale $X=\{X_t \}_{t>0}$, i.e. indexed by $t\in (0,\infty)$, we can define the integral along $X$ on $(\varepsilon, t]$ for $\ep>0$. In particular we denote the quadratic variation of $X$ on $(\varepsilon,t]$ by
\[
[X,X]_t^{\varepsilon}:=\int_{(\varepsilon,t]}d[X,X]_s.
\]
We will use the following notation: For an $\{\mathcal{F}_t\}_{t\geq 0}$-adapted process $H$ and an $\{\mathcal{F}_t\}_{t\geq 0}$-stopping time $\tau$, we write
\[
H^{\tau}_t(\omega)=H_{t\land \tau(\omega)}(\omega)\ \text{for}\ t\geq 0,\ \omega \in \Omega.
\]
\begin{proof}[Proof of \tref{convtto0} (2)]
Since $X_t \in M$ for sufficiently small $t>0$, it suffices to show that $\displaystyle \lim_{t\to 0}\bar{f}(X_t)$ exists almost surely for all $f\in C^{\infty}(M)$ by \cref{cor2}, where $\bar{f}$ is an extension of $f$ satisfying \eqref{nabla}. By It\^o's formula, it holds that
\begin{align}
\bar{f}(X_t)-\bar{f}(X_{\varepsilon})&=\int_{(\varepsilon, t]} \langle D\bar{f}(X_{s-}), dX_s\rangle + \frac{1}{2}\int_{(\varepsilon, t]}\mathrm{Hess}\bar{f}(X_{s-})\, d[X,X]^c_s \nonumber \\
&\h +\sum_{\varepsilon < s \leq t}\{\bar{f}(X_s)-\bar{f}(X_{s-})-\langle D\bar{f}(X_{s-}), \eta(X_{s-},X_s) \rangle \}.\label{itoepsilon}
\end{align}
Since it holds that
\begin{align*}
&\int_{(\varepsilon, t]}|\mathrm{Hess}\bar{f}(X_{s-})|\, |d[X,X]^c_s| \leq \lim_{\varepsilon \to 0}C[X,X]^{c, \varepsilon}_t,\\
&\sum_{\varepsilon < s \leq t}| \bar{f}(X_s)-\bar{f}(X_{s-})-\langle D\bar{f}(X_{s-}), \eta(X_{s-},X_s) \rangle |\leq \lim_{\varepsilon \to 0}C[X,X]^{d, \varepsilon}_t,
\end{align*}
for some $C>0$, we can define
\begin{align*}
A_t&:=\lim_{\varepsilon \to 0} \frac{1}{2}\int_{(\varepsilon,t]}\mathrm{Hess} \bar{f}(X_{s-})\, d[X,X]^c_s,\\
B_t&:=\lim_{\varepsilon \to 0}\sum_{\varepsilon < s \leq t}\{ \bar{f}(X_s)-\bar{f}(X_{s-})-\langle D\bar{f}(X_{s-}), \eta(X_{s-},X_s) \rangle \}
\end{align*}
for $t>0$. We set $H_t:=\bar{f}(X_t)-A_t-B_t$ for $t>0$. Then it holds that
\[
H_t-H_{\varepsilon}=\int_{(\varepsilon, t]} \langle D\bar{f}(X_{s-}), dX_s\rangle,\ 0<\varepsilon \leq t.
\]
Thus $\{H_t\}_{t>0}$ is a local martingale on $(0,\infty)$. Furthermore, since the jumps of $B$ are uniformly bounded by the compactness of $M$, we can take an increasing sequence of stopping times $\tau_1\leq \tau_2 \leq \dots \leq \tau_n \leq \dots \to \infty$ almost surely such that $A^{\tau_n}\mathbf{1}_{\{ \tau_n>0\}}$ and $B^{\tau_n}\mathbf{1}_{\{\tau_n>0\}}$ are bounded. Thus $H^{\tau_n}\mathbf{1}_{\{\tau_n>0\}}$ is a bounded martingale on $(0,\infty)$ for each $n\in \mathbb{N}$ and can be extended to a bounded martingale on $[0,\infty)$. In particular, $\displaystyle \lim_{t\to 0}H_t$ exists almost surely on $\{\tau_n >0 \}$ for each $n$. Thus $\displaystyle \lim_{t\to 0}\bar{f}(X_t)$ exists in $\mathbb{R}$ almost surely and this completes the proof. 
\end{proof}
\begin{proof}[Proof of \tref{convtto0} (1)]
Suppose $(\{X_t\}_{t>0},p,\zeta)$ is an $M$-valued $\eta$-martingale with an end point and the limit $\displaystyle X_0:=\lim_{t\to 0}X_t$ exists in $M$. We will show that $(\{X_t\}_{t\geq 0},p,\zeta)$ is an $M$-valued $\eta$-martingale with an end point. Let $\iota \colon M \to \mathbb{R}^d$ be an inclusion map and $\bar{\iota}=(\bar{\iota^1},\dots,\bar{\iota^d})$ an extension of $\iota$ satisfying \eqref{nabla}. Since $M$ is compact, $a:=\displaystyle \sum_{i=1}^d\sup_{x\in M}\| \mathrm{Hess} \bar{\iota^i} (x)\|$ is finite. Take $0<C < \frac{1}{2d e^2 a}$. First we set
\begin{align*}
\tau&:=\inf \{t\geq 0\mid |\bar{\iota}(X_t)-\bar{\iota}(X_0)|\geq C \},\\
Y_t&:=\frac{1}{C}(\bar{\iota}(X_t)-\bar{\iota}(X_0))^{\tau}.
\end{align*}
Then $|Y_t|\leq 1$ for $0\leq t <\tau$. To show $\{ \bar{\iota}(X_t)\}_{t\geq 0}$ is a semimartingale, we will use
\[
Z^i_t:=e^{Y^i_t},\ Z_t:=\sum_{i=1}^d Z^i_t.
\]
By It\^o's formula, it holds that
\begin{align}
Z^i_{t}-Z^i_{\varepsilon}&=\int_{(\varepsilon, t]}e^{Y^i_{s-}}\, dY^i_s +\frac{1}{2}\int_{(\varepsilon, t]}e^{Y^i_{s-}}\, d[Y^i,Y^i]^c_s \nonumber \\
&\h +\sum_{\varepsilon<s\leq t}\left(e^{Y^i_s}-e^{Y^i_{s-}}-e^{Y^i_{s-}}\Delta Y^i_s \right) \nonumber \\
&=\frac{1}{C} \int_{(\varepsilon, t]}e^{Y^i_{s-}}\, d\bar{\iota^i}(X)^{\tau}_s +\frac{1}{2C^2}\int_{(\varepsilon, t]}e^{Y^i_{s-}}\, d[\bar{\iota^i}(X)^{\tau},\bar{\iota^i}(X)^{\tau}]^c_s \nonumber \\
&\h +\sum_{\varepsilon<s\leq t}\left(e^{Y^i_s}-e^{Y^i_{s-}}-e^{Y^i_{s-}}\Delta Y^i_s \right) \nonumber \\
&=\frac{1}{C} \int_{(\varepsilon, t]}e^{Y^i_{s-}}\, d\left( \int_{(\varepsilon, s]}D_j \bar{\iota ^i}(X_{u-})^{\tau}\, d\bar{\iota^j}(X)^{\tau}_u\right) \nonumber \\
&\h + \frac{1}{2C}\int_{(\varepsilon,t]} e^{Y^i_{s-}} D_jD_k\bar{\iota^i}(X_{s-})^{\tau}\, d[\bar{\iota^j}(X)^{\tau},\bar{\iota^k}(X)^{\tau}]^c_s \nonumber \\
&\h +\frac{1}{2C^2}\int_{(\varepsilon, t]}e^{Y^i_{s-}}\, d[\bar{\iota^i}(X)^{\tau},\bar{\iota^i}(X)^{\tau}]^c_s \nonumber \\
&\h +\frac{1}{C}\sum_{\varepsilon<s\leq t}\{ e^{Y^i_{s-}}\left(\bar{\iota^i}(X_{s})^{\tau}-\bar{\iota^i}(X_{s-})^{\tau}-\langle D\bar{\iota^i}(X_{s-})^{\tau},\Delta \bar{\iota} (X_s)^{\tau} \rangle \right) \} \nonumber \\
&\h + \sum_{\varepsilon<s\leq t}\left(e^{Y^i_s}-e^{Y^i_{s-}}-e^{Y^i_{s-}}\Delta Y^i_s \right).\label{Zt}
\end{align}
Here the first term of the right-hand side of \eqref{Zt} is a local martingale. We further set
\begin{align*}
J^{i, \varepsilon}_t&:=\frac{1}{C}\mathbf{1}_{\{\varepsilon < \tau \leq t\}}\left( e^{Y^i_{\tau-}}\left( \bar{\iota^i}(X_{\tau}) - \bar{\iota^i}(X_{\tau-}) - \langle D\bar{\iota^i}(X_{\tau-}), \Delta \bar{\iota} (X_\tau) \rangle \right)\right),\ t\geq \varepsilon>0,\\
J^{\varepsilon}_t&:=\sum_{i=1}^dJ_t^{i,\varepsilon},\ t\geq \varepsilon>0, \\
K^{i, \varepsilon}_t&:=\mathbf{1}_{\{\varepsilon < \tau \leq t\}}\left( e^{Y^i_{\tau}}-e^{Y^i_{\tau-}}-e^{Y^i_{\tau-}}\Delta Y^i_{\tau} \right),\ t\geq \varepsilon>0, \\
K^{\varepsilon}_t&:=\sum_{i=1}^dK^{i,\varepsilon}_t,\ t\geq \varepsilon>0.
\end{align*}
Then $J^{\varepsilon}$ and $K^{\varepsilon}$ are the jumps at $t=\tau$ of the discontinuous locally bounded variation parts of \eqref{Zt}. We define the two processes of locally finite variation as follows:
\begin{align*}
A^{\varepsilon}_t&=\sum_{i=1}^d\left(\frac{1}{2C}\int_{(\varepsilon,t]} e^{Y^i_{s-}} D_jD_k\bar{\iota^i}(X_{s-})^{\tau}\, d[\bar{\iota^j}(X_{s-}),\bar{\iota^k}(X_{s-})]^c_s\right. \nonumber \\
&\h \left. +\frac{1}{2C^2}\int_{(\varepsilon, t]}e^{Y^i_{s-}}\, d[\bar{\iota^i}(X)^{\tau},\bar{\iota^i}(X)^{\tau}]^c_s \right),\\
B^{\varepsilon}_t&=\sum_{i=1}^d\left( \frac{1}{C}\sum_{\varepsilon<s\leq t}\{ e^{Y^i_{s-}}\left(\bar{\iota^i}(X_{s})^{\tau}-\bar{\iota^i}(X_{s-})^{\tau}-\langle D\bar{\iota^i}(X_{s-})^{\tau},\Delta \bar{\iota} (X_s)^{\tau}\rangle \right) \} \right. \nonumber \\
&\h \left.+ \sum_{\varepsilon<s\leq t}\left(e^{Y^i_s}-e^{Y^i_{s-}}-e^{Y^i_{s-}}\Delta Y^i_s \right) \right)-J^{\varepsilon}_t-K^{\varepsilon}_t.
\end{align*}
Then it holds that for $\varepsilon < s\leq \tau$,
\begin{align}
dA^{\varepsilon}_s&=\sum_{i=1}^d \left( \frac{1}{2C}e^{Y^i_{s-}} D_jD_k\bar{\iota^i}(X_{s-})^{\tau}\, d[\bar{\iota^j}(X)^{\tau},\bar{\iota^k}(X)^{\tau}]^c_s+\frac{1}{2C^2}e^{Y^i_{s-}}\, d[\bar{\iota^i}(X)^{\tau},\bar{\iota^i}(X)^{\tau}]^c_s \right) \nonumber \\
&\geq \sum_{i=1}^d\left( -\frac{ea}{C}  d[\bar{\iota}(X)^{\tau}, \bar{\iota}(X)^{\tau}]^c_s + \frac{e^{-1}}{2C^2} d[\bar{\iota^i}(X)^{\tau},\bar{\iota^i}(X)^{\tau}]^c_s\right) \nonumber \\
&= \left( -\frac{dea}{C} + \frac{e^{-1}}{2C^2} \right)d[\bar{\iota}(X)^{\tau}, \bar{\iota}(X)^{\tau}]^c_s. \label{Ageq}
\end{align}
Furthermore, it holds that for $\varepsilon < s \leq \tau$,
\begin{align}
dB^{\varepsilon}_s &= \sum_{i=1}^d\frac{e^{Y^i_{s-}}}{C}\left(\bar{\iota^i}(X_{s})^{\tau}-\bar{\iota^i}(X_{s-})^{\tau}-\langle D\bar{\iota^i}(X_{s-})^{\tau},\Delta \bar{\iota} (X_s)^{\tau}\rangle \right)\mathbf{1}_{\{s<\tau\}} \nonumber \\
&\h +\sum_{i=1}^d\left(e^{Y^i_s}-e^{Y^i_{s-}}-e^{Y^i_{s-}}\Delta Y^i_s\right)\mathbf{1}_{\{s<\tau\}} \nonumber \\
&\h +\sum_{i=1}^d\frac{e^{Y^i_{\tau-}}}{C} \left(\bar{\iota^i}(X_{\tau})-\bar{\iota^i}(X_{\tau-})-\langle D\bar{\iota^i}(X_{\tau-}),\Delta \bar{\iota} (X_{\tau})\rangle \right)\mathbf{1}_{\{s=\tau\}} \nonumber \\
&\h +\sum_{i=1}^d \left(e^{Y^i_{\tau}}-e^{Y^i_{\tau-}}-e^{Y^i_{\tau-}}\Delta Y^i_s\right)\mathbf{1}_{\{s=\tau\}} -dJ^{\varepsilon}_s-dK^{\varepsilon}_s\nonumber \\
&\geq - \frac{de a}{C}|\Delta \bar{\iota} (X_s)|^2\mathbf{1}_{\{s<\tau\}}\nonumber \\
&\h + \inf \left\{ \frac{e^y-e^x-e^x(y-x)}{(y-x)^2}\mid x,y\in \mathbb{R}, x\neq y,  |x|, |y| \leq 1 \right\} \sum_{i=1}^d|\Delta Y^i_s |^2\mathbf{1}_{\{s<\tau\}} \nonumber \\
& \geq \left( -\frac{de a}{C} + \frac{e^{-1}}{2C^2}\right) |\Delta \bar{\iota}(X_s)|^2\mathbf{1}_{\{s<\tau\}}. \label{Bgeq}
\end{align}
Thus $A^{\ep}$ and $B^{\ep}$ are increasing processes by our choice of $C$, i.e. $0<C<\frac{1}{2de^2a}$. We further set
\begin{align*}
J_t&:=\sum_{i=1}^d\mathbf{1}_{\{\tau \leq t\}}\left( e^{Y^i_{\tau-}}\left( \bar{\iota^i}(X_\tau) - \bar{\iota^i}(X_{\tau-}) - \langle D\bar{\iota^i}(X_{\tau-}), \Delta \bar{\iota} (X)_{\tau} \rangle \right)\right),\ t\geq 0,\\
K_t&:=\sum_{i=1}^d\mathbf{1}_{\{\tau \leq t\}}\left( e^{Y^i_{\tau}}-e^{Y^i_{\tau-}}-e^{Y^i_{\tau-}}\Delta Y^i_{\tau} \right),\ t\geq 0.
\end{align*}
Then by \eqref{Ageq} and \eqref{Bgeq}, we can deduce that $Z_t - J_t - K_t$ is a bounded submartingale on $[\varepsilon,\infty)$ for all $\varepsilon>0$ since it holds that
\begin{align}
Z_t-J_t-K_t &= Z_t -J_{\varepsilon} -J^{\varepsilon}_t - K_{\varepsilon} -K^{\varepsilon}_t \nonumber \\
&= \frac{1}{C} \int_{(\varepsilon, t]}e^{Y^i_{s-}}\, d\left( \int_{(\varepsilon, s]}D_j \bar{\iota ^i}(X_{u-})^{\tau}\, d\bar{\iota^j}(X)^{\tau}_u\right)\nonumber \\
&\h +A^{\varepsilon}_t+B^{\varepsilon}_t-J_{\varepsilon}-K_{\varepsilon}\label{Zt2}
\end{align}
for $t\geq \varepsilon>0$. Thus $Z-J-K$ can be extended to a bounded submartingale on $[0,\infty)$ and there exist a uniformly integrable martingale $L$ and a predictable integrable increasing process $R$ with $R_0=0$ such that
\[
Z_t-J_t-K_t=L_t+R_t
\]
by the Doob-Meyer decomposition. On the other hand, since the first term of \eqref{Zt2} is a local martingale with upper bound and equal to zero at $t=\varepsilon$, it is a submartingale and in particular integrable for each $t\geq \varepsilon$. Thus both $A^{\ep}_t$ and $B^{\ep}_t$ are integrable processes for each $t\geq \varepsilon$. Moreover there exists a predictable integrable increasing process $\tilde B^{\varepsilon}$ with $\tilde B^{\varepsilon}_{\ep}=0$ such that $B^{\varepsilon}-\tilde B^{\varepsilon}$ is a martingale. Therefore it holds that
\[
R_t-R_{\varepsilon}=A^{\varepsilon}_t + \tilde B^{\varepsilon}_t,\ t\geq \varepsilon
\]
by the uniqueness of the Doob-Meyer decomposition. Thus we obtain
\[
\mathbb{E}\left[ \lim_{\varepsilon \to 0}A^{\varepsilon}_t \right]<\infty,\ \mathbb{E}\left[ \lim_{\varepsilon \to 0}\tilde B^{\varepsilon}_t\right] <\infty.
\]
Since $B^{\varepsilon}_t-\tilde B^{\varepsilon}_t$ is a martingale, it holds that
\[
\mathbb{E}\left[ \lim_{\varepsilon \to 0}B^{\varepsilon}_t \right] =\mathbb{E} \left[ \lim_{\varepsilon \to 0}\tilde B^{\varepsilon}_t \right] <\infty.
\]
Thus $\displaystyle \lim_{\varepsilon \to 0}[\bar{\iota} (X)^{\tau}, \bar{\iota}(X)^{\tau}]_t^{\varepsilon}<\infty$ a.s. by \eqref{Ageq} and \eqref{Bgeq}. Moreover, for fixed $t>0$ and all $n\in \mathbb{N}$, it holds that
\begin{align*}
\lim_{\varepsilon \to 0}\mathbf{1}_{\{\frac{t}{n}<\tau \leq \frac{t}{n-1}\}}[\bar{\iota}(X),\bar{\iota}(X)]^{\varepsilon}_t&= \lim_{\varepsilon \to 0}\mathbf{1}_{\{\frac{t}{n}<\tau \leq \frac{t}{n-1}\}}\left([\bar{\iota}(X),\bar{\iota}(X)]^{\varepsilon}_{\tau \land \frac{t}{n}}+[\bar{\iota}(X),\bar{\iota}(X)]^{\frac{t}{n}}_t\right)\\
&=\lim_{\varepsilon \to 0}\mathbf{1}_{\{\frac{t}{n}<\tau \leq \frac{t}{n-1}\}}\left([\bar{\iota}(X)^{\tau},\bar{\iota}(X)^{\tau}]^{\varepsilon}_{\frac{t}{n}}+[\bar{\iota}(X),\bar{\iota}(X)]^{\frac{t}{n}}_t\right)\\
&<\infty
\end{align*}
almost surely, where we use the convention that $\displaystyle \frac{t}{0}=\infty$. Since $\tau>0$ a.s., we obtain
\begin{gather}\label{quadX}
\lim_{\varepsilon \to 0}[\bar{\iota}(X),\bar{\iota}(X)]^{\varepsilon}_t<\infty\ \text{a.s.}
\end{gather}
By \eqref{quadX}, it follows that
\begin{align*}
&\lim_{\varepsilon \to 0}\int_{(\varepsilon,t]}|D_jD_k\bar{\iota} (X_{s-})|\, |d[\bar{\iota^j}(X),\bar{\iota^k}(X)]^c_s| <\infty,\\
&\lim_{\varepsilon \to 0}\sum_{\varepsilon <s\leq t}|\bar{\iota}(X_s) - \bar{\iota}(X_{s-}) - \langle D\bar{\iota} (X_{s-}, \Delta \bar{\iota}(X_s)\rangle |<\infty
\end{align*}
almost surely. Thus we can set
\begin{align*}
N^i_t&:=\bar{\iota}(X_t)-\bar{\iota}(X_0)-\lim_{\varepsilon \to 0}\int_{(\varepsilon,t]}D_jD_k\bar{\iota} (X_{s-})\, d[\bar{\iota^j}(X),\bar{\iota^k}(X)]^c_s\\
&\h -\lim_{\varepsilon \to 0}\sum_{\varepsilon <s\leq t}\left(\bar{\iota}(X_s) - \bar{\iota}(X_{s-}) - \langle D\bar{\iota} (X_{s-}, \Delta \bar{\iota}(X_s)\rangle \right).
\end{align*}
Then it holds that
\begin{align*}
N_t-N_{\varepsilon}=\int_{(\varepsilon,t]}\langle D \bar{\iota}(X_{s-}), d\bar{\iota}(X_s)\rangle
\end{align*}
for $0<\varepsilon \leq t$ by It\^o's formula. Thus $\{N_t\}_{t\geq \varepsilon}$ is a local martingale. Furthermore, in the same way as the proof of \tref{convtto0} (2), we can take a sequence of stopping times $\{ \tau_n\}$ such that $\tau_n\to \infty$ as $n\to \infty$ a.s. and $N^{\tau_n}\mathbf{1}_{\{\tau_n>0\}}$ is bounded. Therefore $\{N_t^{\tau_n}\mathbf{1}_{\{\tau_n>0\}} \}_{t\geq \varepsilon}$ is a bounded martingale for each $\varepsilon>0$ and consequently we can deduce that $N^{\tau_n}\mathbf{1}_{\{\tau_n>0\}}$ is a bounded martingale. In particular, $\{N_t\}_{t\geq 0}$ is a local martingale and $\{\bar{\iota}(X_t)\}_{t\geq 0}$ is a semimartingale. Furthermore, we can define the stochastic integral $\int_{0}^t \langle D\bar{\iota} (X_{s-}), d\bar{\iota}(X_s)\rangle$ and this equals $N$. Thus $\{X_t\}_{t\geq 0}$ is an $\eta$-martingale with an end point by \lref{martingalelem1}.
\end{proof}
\section{Applications to harmonic maps}\label{harmonic}
We will consider applications of martingales with jumps to harmonic maps with respect to non-local Dirichlet forms. First we briefly review the theory of Dirichlet forms and Markov processes. See \cite{FOT} and \cite{ChenFuku} for details.

Let $E$ be a locally compact and separable metric space. We add a point $\Delta$ to $E$ and define $E_{\Delta}:=E\cup \{\Delta \}$. Given $f\in \mathcal{B}(E)$, we extend the domain of $f$ to $E_{\Delta}$ by $f(\Delta)=0$. Let $m$ be a positive Radon measure with full support on $E$. Let $\mathcal{F}$ be a dense subspace of $L^2(E;m)$ and $\mathcal{E}\colon \mathcal{F} \times \mathcal{F} \to \mathbb{R}$ a non-negative definite symmetric quadratic form. We set
\[
\mathcal{E}_1(u,v):=\mathcal{E}(u,v)+\langle u,v \rangle_{L^2}
\]
for $u,v\in \mathcal{F}$. The pair $(\mathcal{E}, \mathcal{F})$ is called a \textbf{Dirichlet form} if the space $(\mathcal{F}, \mathcal{E}_1)$ is a Hilbert space and it holds that $(u\lor 0)\land 1 \in \mathcal{F}$ and
\[
\mathcal{E}\left( (u\lor 0)\land 1, (u\lor 0)\land 1 \right)\leq \mathcal{E}(u,u)
\]
for any $u\in \mathcal{F}$. A Dirichlet form $(\mathcal{E},\mathcal{F})$ is called \textbf{regular} if $\mathcal{F}\cap C_0(E)$ is $\mathcal{E}_1$-dense in $\mathcal{F}$ and $\|\cdot \|_{\infty}$-dense in $C_0(E)$. 
For a regular Dirichlet form, the \textbf{extended Dirichlet space} $\mathcal{F}_e$ is defined as the family of equivalence classes of Borel functions $u\colon E\to \mathbb{R}$ with respect to the $m$-a.e. equality such that $|u|<\infty\ m$-a.e. and there exists an $\mathcal{E}$-Cauchy sequence $\{u_k\}_{k=1}^{\infty}$ in $\mathcal{F}$ such that $\displaystyle \lim_{k\to \infty}u_k=u,\ m$-a.e. For $u\in \mathcal{F}_e$, we can set
\[
\mathcal{E}(u,u):=\lim_{k\to \infty}\mathcal{E}(u_k,u_k),
\]
where $\{u_k\}_{k=1}^{\infty}$ is an $\mathcal{E}$-approximating sequence in $\mathcal{F}$ of $u$. For a regular Dirichlet form $(\mathcal{E},\mathcal{F})$, there exists an $m$-symmetric Hunt process $(\Omega, \{Z\}_{t\geq 0},\{ \theta_t\}_{t\geq 0}, \zeta, \{ \mathbb{P}_z\}_{z\in E_{\Delta}})$ on $E$ corresponding to $(\mathcal{E},\mathcal{F})$, where $\theta_t:\Omega \to \Omega$ is the shift operator satisfying for all $\omega \in \Omega$,
\begin{gather*}
Z_s\circ \theta_t(\omega)=Z_{s+t}(\omega),\ \theta_0\omega=\omega,\\
\theta_{\infty}\omega(t)=\Delta,\ \text{for all}\ t\geq 0 ,
\end{gather*}
$\zeta$ is a life time of $Z$ and $\mathbb{P}_z$ is the distribution of $\{Z_t\}$ stating at $z$. We denote the transition function of $Z$ by $\{ p_t\}_{t\geq 0}$. The $0$-resolvent of $Z$ is defined by
\[
Rf(z):=\int_0^t p_tf(z)\, dt
\]
for $f\in \mathcal{B}_+(E)$, where $\mathcal{B}_+(E)$ is the set of all non-negative Borel functions on $E$. Let
\begin{align*}
\mathcal{F}^0_{\infty}&=\sigma(Z_s;\ s< \infty),\\
\mathcal{F}^0_t&=\sigma(Z_s;\ s\leq t).
\end{align*}
For a $\sigma$-finite measure $\mu$, we write
\[
\mathbb{P}_{\mu}(\Lambda)=\int_{E_{\Delta}}\mathbb{P}_z(\Lambda)\, \mu(dz),\ \Lambda \in \mathcal{F}^0_{\infty}.
\]
In particular, if $\mu$ is a probability measure on $E$, $\mathbb{P}_{\mu}$ is a probability measure on $\mathcal{F}^0_{\infty}$. Denote the $\mathbb{P}_{\mu}$-completion of $\mathcal{F}^0_{\infty}$ by $ \mathcal{F}^{\mu}_{\infty}$. Let
\[
\mathcal{F}^{\mu}_t=\sigma(\mathcal{F}^0_t,\mathcal{N}_{\mu}),
\]
where $\mathcal{N}_{\mu}$ is the family of all $\mathbb{P}_{\mu}$-null sets in $\mathcal{F}^{\mu}_{\infty}$. Denote the set of probability measures on $E_{\Delta}$ by $\mathcal{P}(E_{\Delta})$ and let
\[
\mathcal{F}^Z_t=\bigcap_{\mu \in \mathcal{P}(E_{\Delta})} \mathcal{F}^{\mu}_t,\ t\in [0,\infty].
\]
The filtration $\{ \mathcal{F}^Z_t \}_{t\geq 0}$ is called the \textbf{minimum augmented admissible filtration} of $Z$. A subset $A\subset E_{\Delta}$ is called a \textbf{nearly Borel set} if for every $\mu \in \mathcal{P}(E_{\Delta})$, there exist Borel sets $A_0$, $A_1$ in $E_{\Delta}$ such that
\[
A_0\subset A\subset A_1,\ \mathbb{P}_{\mu}(Z_t\in A_1\backslash A_0,\ \text{for some }t\geq 0)=0.
\]
For a subset $A\subset E_{\Delta}$, define the stopping time $\sigma_A$ by
\begin{gather*}
\sigma_{A}(\omega):=\inf \{ t>0\mid Z_t(\omega)\in A\},\ \omega \in \Omega.
\end{gather*}
It is known that if $A$ is a nearly Borel set, then $\sigma_A$ is an $\{ \mathcal{F}^Z_t \}$-stopping time. A subset $A\subset E$ is said to be \textbf{$\bm{m}$-polar} if there exists a nearly Borel set $A_2$ such that $A\subset A_2$ and
\begin{gather*}
\mathbb{P}_m(\sigma _{A_2}<\infty)=\int_E\mathbb{P}_z(\sigma_{A_2}<\infty)\, m(dz)=0.
\end{gather*}
A property holds ``quasi-everywhere'', q.e. for short, if it holds except for an $m$-polar set. A subset $A\subset E$ is said to be \textbf{finely open} if for all $z\in A$, there exists a nearly Borel set $D_z\supset E\, \backslash \, A$ such that
\[
\mathbb{P}_z(\sigma_{D_z}>0)=1.
\]
We use the abbreviation AF (resp., CAF) for additive functional (resp., continuous additive functional). See \cite{ChenFuku} or \cite{FOT} for details about AF's. We denote the energy of an AF $A$ by
\[
\mathbf{e}(A)=\lim_{t\searrow 0}\frac{1}{2t}\mathbb{E}_m\left[ A_t^2 \right].
\]
We set
\begin{align*}
\mathcal{M}:=\{ M \mid M\ \text{is an AF},\ \mathbb{E}_z\left[ M_t^2 \right] <\infty,\ \mathbb{E}_z \left[ M_t \right] =0,\ \text{for q.e.}\ z\in E,\ t\geq 0\}.
\end{align*}
An additive functional in $\mathcal{M}$ is called a \textbf{martingale additive functional (MAF)}. For details about MAF's, see \cite{ChenFuku} of \cite{FOT}. MAF's often appear as follows. For $u\in \mathcal{F}_e$, there exist an MAF $M^{[u]}$ with finite energy and a CAF $N^{[u]}$ with zero energy such that
\begin{align}\label{FukuDec}
\tilde u(Z_t)-\tilde u(Z_0)=M^{[u]}_t+N^{[u]}_t,\ \mathbb{P}_z\text{-a.s. for q.e.}\ z\in E
\end{align}
and this expression is unique, where $\tilde u$ is a quasi-continuous modification of $u$. See \cite{Fuku} or Theorem 5.2.2 of \cite{FOT} for details. We define a \textbf{local MAF} as an AF which is a $\mathbb{P}_z$-local martingale for q.e. $z\in E$. For simplicity we suppose that a Dirichlet form $(\mathcal{E},\mathcal{F})$ is either recurrent or transient, that is the corresponding Hunt process $Z$ is either recurrent or transient, respectively. We will use the \textbf{reflected Dirichlet space} of $(\mathcal{E},\mathcal{F})$ denoted by $\mathcal{F}^r$. Historically, the reflected Dirichlet space was introduced in \cite{Silver}. Our main reference however is \cite{ChenFuku}. If $(\mathcal{E},\mathcal{F})$ is recurrent, it holds that $\mathcal{F}^r=\mathcal{F}_e$. On the other hand, if $(\mathcal{E},\mathcal{F})$ is transient, then the reflected Dirichlet space is defined as follows: A random variable $\phi$ on $(\Omega, \mathcal{F}_{\infty}^Z)$ is said to be \textbf{terminal} if for q.e. $z\in M$, $\phi$ is $\mathbb{P}_z$-integrable and it satisfies
\[
\left\{
\begin{array}{ll}
\{Z_{\zeta-}\in E,\ \zeta <\infty\} \cup \{\zeta=0\}\subset \{ \phi=0\},\\
\phi\circ \theta_t=\phi,\ t<\zeta
\end{array}
\right.
\]
$\mathbb{P}_z$-a.s. q.e. $z\in E$. For a terminal random variable $\phi$, define $h_{\phi}:M\to \mathbb{R}$ by $h_{\phi}(z)=\mathbb{E}_z\left[ \phi \right]$ and let
\begin{eqnarray}\label{Mphi}
M^{\phi}_t=
\begin{cases}
h_{\phi}(Z_t)-h_{\phi}(Z_0),\ &t<\zeta,\\
\phi-h_{\phi}(Z_0), \ &t\geq \zeta.
\end{cases}
\end{eqnarray}
Then $M^{\phi}$ is a $(\mathbb{P}_z,\{ \mathcal{F}_t^Z\}_{t\geq 0})$-uniformly integrable martingale, q.e. $z$. We let
\begin{align*}
\mathbf{N}=\{ \phi \mid &\ \text{a terminal random variable with}\ \mathbb{E}_z \left[ \phi^2 \right] <\infty,\ \text{q.e.}\ z\in M\ \text{and}\\
&M^{\phi}\ \text{is a martingale additive functional (MAF) with finite energy}\}, \\
\mathbf{HN}=\{h_{\phi}& \mid \phi \in \mathbf{N}\}.
\end{align*}
Then the reflected Dirichlet space is defined by
\[
\mathcal{F}^r=\mathcal{F}_e+\mathbf{HN}.
\]
The quadratic form $\mathcal{E}^r$ on the reflected Dirichlet space of a transient Dirichlet space is defined by
\[
\mathcal{E}^r(u,u)=\mathcal{E}(f,f)+\mathbf{e}(M^{\phi})+\frac{1}{2}\int_E |h_{\phi}(z)|^2\, k(dz), 
\]
where $u=f+h_{\phi}$, $f\in \mathcal{F}_e$, $\phi \in \mathbf{H}$, and $k$ is the killing measure on $E$, which is defined by the vague limit of the family $\{ \frac{1}{t_n}(1-p_{t_n}1)\cdot m \}_{n \in \mathbb{N}}$ of measures for a subsequence $\{t_n\}_{n\in \mathbb{N}}$ satisfying $t_n \downarrow 0$ as $n \to \infty$. First we consider the relation between the analytic characterization of harmonic maps and the probabilistic one in some situations. In \pref{harmonicmartingale}, we will show the relationship between a manifold-valued martingale and a harmonic map contained in the domain of the generator of a Markov process. To show such a result for general harmonic maps which belong to local Dirichlet forms, we need to consider the stochastic integral along additive functionals with zero energy, which was considered in \cite{Nakao}, \cite{CFKZ} and \cite{Kuwae} and extend the result of \cite{Pic4} to the case of non-local Dirichlet forms. 
However, in this article, we only show such a result in particular cases in order to make this article concise and focus on applications of the theory of martingales.
\begin{dfn}
The operator $\mathcal{L}:\mathcal{F}^r \to L^2(E;m)$ is defined by 
\begin{align*}
\mathcal{D}(\mathcal{L})&=\{ u\in \mathcal{D}^r\mid \text{there exists}\ w\in L^2(E;m)\ \text{such that}\ \\
&\hspace{2cm} \mathcal{E}^r(u,v)=-\langle w,v\rangle_{L^2},\ \text{for all}\ v\in \mathcal{F}\cap C_0(E) \},\\
\mathcal{L}u&=w,\ u\in \mathcal{D}(\mathcal{L}).&
\end{align*}
\end{dfn}
Let $M$ be a compact Riemannian submanifold of $\mathbb{R}^d$ with $\text{dim}\, M=n$ and $\iota:M\to \mathbb{R}^d$ an embedding. Denote the orthonormal projection from $\mathbb{R}^d$ to $T_yM$ by $\Pi_y$ for $y\in M$. Let $\eta$ be a connection rule on $M$ defined by
\[
\eta(x,y)=\Pi_x(y-x),\ x,y\in M.
\]
Let
\[
\mathcal{D}_M(\mathcal{L})=\{ u=(u^1,\dots,u^d):E\to \mathbb{R}^d\mid u^i\in \mathcal{D}(\mathcal{L})\ \text{for each}\ i\ \text{and}\ \tilde u(z)\in M,\ \text{q.e. }z\}.
\]
\begin{prop}\label{harmonicmartingale}
Suppose that ($\mathcal{E},\mathcal{F}$) is either recurrent or transient. Let $u$ be a Borel measurable map in $\mathcal{D}_M(\mathcal{L})$. Then for any relatively compact open subset $D$, $\{ \tilde u(Z_{t\land \tau_{D}})\}_{t\geq 0}$ is an $M$-valued $(\mathbb{P}_z, \{ \mathcal{F}^Z_t \})$-semimartingale with the end point $0$ for q.e. $z\in E$ and for all $f\in C^{\infty}(M)$,
\[
\int \langle D\bar{f}(\tilde u(Z_-)^{\tau_D}), d\tilde u(Z)^{\tau_D}_s\rangle-\int \langle D\bar{f}(\tilde u(Z_-)^{\tau_D}),\Pi_{\tilde u(Z_-)^{\tau_D}}\mathcal{L}u(Z_-)^{\tau_D}\rangle \, ds
\]
is a $\mathbb{P}_z$-square integrable martingale, where $\bar f$ a function on $\mathbb{R}^d$ satisfying \eqref{nabla}.
\end{prop}
\begin{proof}
Since $u^i\in \mathcal{F}^r$ for each $i=1,\dots,d$, there exist $v^i\in \mathcal{F}_e$ and $\phi^i\in \mathbf{H}$ such that $u^i=v^i+h_{\phi^i}$. Set
\[
M^i_t=\tilde u^i(Z_t)-\tilde u^i(Z_0)-\int_0^t\mathcal{L}u^i(Z_{s-})\, ds + \phi^i \mathbf{1}_{\{t\geq \zeta\}},\ i=1,\dots ,d.
\]
Then $M^i$ is an MAF because it can be written as the sum $M^i=M^{[v^i]}+M^{\phi^i}$, where $M^{[v^i]}$ and $M^{\phi^i}$ are MAF's defined in \eqref{FukuDec} and \eqref{Mphi}, respectively. Since $\phi$ is a terminal random variable and $D$ is relatively compact, it holds that
\[
\phi \mathbf{1}_{\{t\land \tau_D\geq \zeta\}}=0.
\]
Thus
\[
M^i_{t\land \tau_D}=\tilde u^i(Z_{t\land \tau_D})-\tilde u^i(Z_0)-\int_0^{t\land \tau_D}\mathcal{L}u^i(Z_{s-})\, ds
\]
is a $\mathbb{P}_z$-square integrable martingale for q.e. $z\in E$. In particular, $\tilde u(Z)^{\tau_D}$ is a $\mathbb{P}_z$-semimartingale for q.e. $z\in E$. Let $f\in C^{\infty}(M)$. Then we have
\begin{align*}
\int \langle D\bar f(\tilde u(Z_-)^{\tau_D}),d\tilde u(Z)^{\tau_D}_s\rangle&=\int \langle D \bar f(\tilde u(Z_-)^{\tau_D}),dM^{\tau_D}_s\rangle+ \int \langle  D\bar f(\tilde u(Z_-)^{\tau_D}),\mathcal{L}u(Z_-)^{\tau_D}\rangle\, ds\\
=\int \langle  D\bar f(\tilde u(Z_-)^{\tau_D})&,dM^{\tau_D}_s\rangle + \int \langle  D\bar f(\tilde u(Z_-)^{\tau_D}),\Pi_{\tilde u(Z_-)^{\tau_D}}\mathcal{L}u(Z_-)^{\tau_D}\rangle \, ds.
\end{align*}
Thus the required result follows.
\end{proof}
\begin{rem}
Suppose $D$ is a relatively compact open set and satisfies $\mathbb{P}_z(\tau_D<\infty)=1$ for q.e. $z\in D$. By \pref{harmonicmartingale}, if $u\in \mathcal{D}_M(\mathcal{L})$ satisfies
\[
\mathcal{L}u(z)\perp T_{u(z)}N,\ m\text{-a.e. }z\in D, 
\]
then $\tilde u(Z)^{\tau_D}$ is an $M$-valued $(\mathbb{P}_z, \{ \mathcal{F}^Z_t \})$-$\eta$-martingale with the end point $0$ for q.e. $z \in D$. Indeed, if we set
\[
q(z)=\Pi_{\tilde u(z)}\mathcal{L}u(z),\ q(z)=(q^1(z),\dots,q^d(z)),\ z\in E,
\]
it holds that
\begin{align*}
\mathbb{E}_z \left[ \int_0^{\tau_D}|q^i(Z_s)|\, ds \right] &=R^D|q^i|(z)=0,\ \text{q.e.}\ z\in D,\ i=1,\dots, d,
\end{align*}
where $R^D$ is the $0$-resolvent of the part process $(\Omega, \{Z^{D}_t\}_{t\geq 0}, \zeta^{D}, \{\mathbb{P}_z\}_{z\in D})$ of $Z$ on $D$, which is defined by
\begin{align*}
Z^{D}_t(\omega)&=\left\{
\begin{array}{ll}
Z_t(\omega),\ &0\leq t<\tau_{D},\\
\Delta,\ &t\geq \tau_{D},
\end{array}
\right.\\
\zeta^{D}(\omega)&=\tau_{D}(\omega)
\end{align*}
for $\omega \in \Omega$, since $q^i(z)=0$, a.e. $z\in D$ and $R^D|q^i|$ is quasi-continuous. Therefore there exists an $m$-polar set $A$ such that for all $z\in D\, \backslash \, A$, we have $q^i(Z_{s\land \tau_D-})=0$, a.e. $s\geq 0$, $\mathbb{P}_z$-a.s. Therefore for all $z\in D\, \backslash \, A$ and $f\in C^{\infty}(M)$,
\[
\langle D\bar{f}(\tilde u(Z_{s-})^{\tau_D}),q(Z_{s-})^{\tau_D}\rangle=0,\ \text{a.e.}\ s\geq 0,\ \mathbb{P}_z\text{-a.s.},
\]
and consequently
\[
\int_0^t\langle D\bar{f}(\tilde u(Z_s)^{\tau_D}),q(Z_s)^{\tau_D}\rangle \, ds=0,\ \mathbb{P}_z\text{-a.s.}
\]
Thus $\int \langle D\bar{f} (\tilde u(Z_-)^{\tau_D})\, d\tilde u(Z)^{\tau_D}\rangle$ is a $\mathbb{P}_z$-square integrable martingale for $z\in D\, \backslash \, A$ by \pref{harmonicmartingale}. In particular, there exist $m$-polar sets $N_i,\ i=1,\dots,d$ such that $\tilde u^i(Z_t)^{\tau_D}=\bar{\iota^i}(\tilde u(Z_t)^{\tau_D})$ is a $(\mathbb{P}_z,\{ \mathcal{F}^Z_t\})$-semimartingale with the end point $0$ and $\int \langle D\bar{\iota^i} (\tilde u(Z_-)^{\tau_D}), d\tilde u(Z)^{\tau_D}\rangle$ is a $(\mathbb{P}_z,\{ \mathcal{F}^Z_t\})$-local martingale for all $z\in D\, \backslash N_i$. Setting $N:=\bigcup_{i=1}^dN_i$, we can deduce that $\tilde u(Z)^{\tau_D}$ is an $M$-valued $(\mathbb{P}_z,\{ \mathcal{F}^Z_t\})$-$\eta$-martingale with the end point $0$ for all $z\in D\, \backslash N$ by \lref{martingalelem1}.
\end{rem}
Following \pref{harmonicmartingale} and \cite{Pic3, Pic4}, we define harmonic maps with respect to Markov processes.
\begin{dfn}
Let $u\colon E\to M$ be a Borel measurable map, $D\subset E$ an open set. Then $u$ is called a \textbf{quasi-harmonic map} on $D$ if for all relatively compact open set $D_1$ with $\overline{D_1}\subset D$, $\{u(Z_{t\land \tau_{D_1}}) \}_{t\geq 0}$ is a $\mathbb{P}_z$-$\eta$-martingale with the end point $0$ for q.e. $z\in E$.
\end{dfn}
\begin{ex}\label{fractional}
We consider the case that $E=\mathbb{R}^m$. Then $(\mathcal{E},\mathcal{F})$ is given by
\[
\begin{cases}
\mathcal{F} = H^{\frac{\alpha}{2}}(\mathbb{R}^m):=\{u\in L^2(\mathbb{R}^m)\mid \mathcal{E}(u,u)<\infty \},\\
\mathcal{E}(u,v)=\int_{\mathbb{R}^m\times \mathbb{R}^m} (u(z)-u(w))(v(z)-v(w))\, N(z,dw)dz,
\end{cases}
\]
where $\alpha \in (0,2)$ and $N(z,dw)$ is a kernel on $\mathbb{R}^m$ given by
\begin{align*}
&N(z,dw)=c_{\alpha, m}|z-w|^{-(m+\alpha)}dw,\\
&c_{m,\alpha}=\alpha 2^{\alpha-2}\pi^{-\frac{m+2}{2}}\sin \left( \frac{\alpha \pi}{2}\right) \Gamma \left( \frac{m+\alpha}{2}\right)  \Gamma \left( \frac{\alpha}{2}\right).
\end{align*}
The Hunt process $Z$ with respect to the above Dirichlet form is called a symmetric $\alpha$-stable process on $\mathbb{R}^m$. This process is a Lévy process with characteristics $(0,\nu,0)$, where $\nu$ is a $\sigma$-finite measure on $\mathbb{R}^m$ given by
\begin{gather*}
\nu(dz)=\frac{c_{m,\alpha}}{|z|^{m+\alpha}}dz.
\end{gather*}
Let $M$ be an $n$-dimensional compact submanifold of $\mathbb{R}^d$. In \cite{Lio, Lio2}, an $\alpha$-fractional harmonic map is defined by the variational problem. The \textbf{local Sobolev space} $H^{\frac{\alpha}{2}}_{loc}(\mathbb{R}^m)$ is the set of functions $u\colon \mathbb{R}^m\to \mathbb{R}$ such that for any relatively compact open subset $G \subset \mathbb{R}^m$, there exists $u_{G} \in H^{\frac{\alpha}{2}}(\mathbb{R}^m)$ such that $u=u_G$, $m$-a.e. on $G$. We set
\[
H_{loc}^{\frac{\alpha}{2}}(\mathbb{R}^m;M)=\{u=(u^1,\dots,u^d):\mathbb{R}^m \to \mathbb{R}^d \mid u^i \in H_{loc}^{\frac{\alpha}{2}}(\mathbb{R}^m)\, \text{for each }i\, \text{and }u(z) \in M,\ \text{a.e. }z \}.
\]
For $u\in H_{loc}^{\frac{\alpha}{2}}(\mathbb{R}^m;M)$, the energy of $u$ is defined by
\[
\mathcal{E}(u)=\int_{\mathbb{R}^m\times \mathbb{R}^m} |u(z)-u(w)|^2\, N(z,dw)dz=\int_{\mathbb{R}^m}|(-\Delta)^{\frac{\alpha}{4}} u(z)|^2dz.
\]
A map satisfying the corresponding Lagrange equation in a weak sense is called a \textbf{weakly $\dfrac{\bm{\alpha}}{\bm{2}}$-fractional harmonic map}. For a bounded domain $D$, the Lagrange equation on $D$ can be written as
\[
(-\Delta)^{\frac{\alpha}{2}}u \perp T_uM
\]
in the sense of distributions. By \pref{harmonicmartingale}, a fractional harmonic map defined by a variational problem is related to a martingale on the target compact submanifold of $\mathbb{R}^d$ at least in the case $u$ is in the domain of the fractional Laplacian.
\end{ex}
We apply \tref{conv2} to $S^1$-valued martingales. For $0<R<\frac{\pi}{2}$, we set
\[
\mathcal{B}_{R}=\{(\cos \theta,\sin \theta)\in S^1\mid |\theta|\leq R  \}.
\]
\begin{lem}\label{F>C}
We fix $R>0$ satisfying $R<\dfrac{\pi}{2}$ and define a function $\phi \colon [-R,R]\to \mathbb{R}$ with a parameter $q>0$ as
\begin{gather}
\phi(t):=1-\cos qt\label{fdef}
\end{gather}
and set
\begin{gather*}
F(s,t):=\frac{\phi (t)-\phi (s)-\phi ' (s)\sin (t-s)}{|t-s|^2}.
\end{gather*}
Then for sufficiently small $q>0$, there exists $C>0$ such that it holds that
\begin{gather}
F(s,t)>C\label{Fineq}
\end{gather}
for all $s,t\in [-R,R]$.
\end{lem}
\begin{proof}
First, for $s,t\in [-R,R]$, there exists $\theta_{st}$ between $s$ and $t$ such that
\[
\phi(t)-\phi(s)-\phi'(s)\sin (t-s) = \frac{q^2\cos q\theta_{st}}{2} \, (t-s)^2 + q\sin (qs) ((t-s)-\sin (t-s)).
\]
Thus it holds that
\begin{align*}
\lim_{|t-s|\to 0}F(s,t)&= \lim_{|t-s|\to 0}\left(\frac{q^2\cos (q\theta_{st})}{2} +q\sin (qs) \frac{(t-s)-\sin (t-s)}{|t-s|^2}  \right)\\
&\geq \frac{q^2\cos (qR)}{2}\\
&>0.
\end{align*}
Therefore there exists $\varepsilon>0$ such that for all $s,t\in [-R,R]$
\[
|t-s|<\varepsilon \Rightarrow F(s,t)>\frac{q^2 \cos (qR)}{4}.
\]
On the other hand, if $|t-s|\geq \varepsilon$, it holds that
\begin{align*}
\frac{F(s,t)}{q^2}&=\frac{1}{(t-s)^2}\left( \frac{\cos (q\theta_{st})(t-s)^2}{2}+\frac{\sin (qs)}{q}((t-s)-\sin (t-s))\right)\\
&\to \frac{t^2-s^2-2s\sin (t-s)}{2|t-s|^2}\ \text{(as}\ q\to 0\ \text{)}\\
&\geq \frac{t^2-s^2-2s\sin (t-s)}{2R^2}
\end{align*}
and the above convergence is uniform. Furthermore it holds that
\[
m:=\min \{ t^2-s^2-2s\sin (t-s)\mid s,t\in [-R,R],\ |t-s|\geq \varepsilon  \}>0.
\]
Thus by taking a sufficiently small $q>0$, $F(s,t)$ can be bounded from below by the constant $\min \{\frac{mq^2}{4R^2}, \frac{q^2 \cos (qR)}{4}\}$.
\end{proof}
\begin{prop}\label{conv3}
Any $\eta$-martingale with values in $\mathcal{B}_R$ converges almost surely in $\mathcal{B}_R$ as $t \to \infty$.
\end{prop}
\begin{proof}
Let $X$ be a $\mathcal{B}_R$-valued martingale. We take $q>0$ satisfying \eqref{Fineq} and define a function $f$ on $\mathcal{B}_R$ by
\begin{gather*}
f(\cos \theta,\sin \theta):=\phi (\theta),
\end{gather*}
where $\phi$ is a function defined in \lref{F>C}. By It\^o's formula, it holds that
\begin{align*}
f(X_t)-f(X_0)&=\int_0^tdf(X_{s-})\, \eta dX_s+\int_0^t\nabla df(X_s)\, d[X,X]^c_s\\
&\h +\sum_{0<s\leq t}\{f(X_s)-f(X_{s-})-\langle{df(X_{s-}),\eta(X_{s-})\rangle }.
\end{align*}
Here by our choice of $\eta$, we have
\begin{align*}
f(y)-f(x)-\langle df(x),\eta (x,y) \rangle&=\phi (\arg (y))-\phi (\arg (x))-\phi'(\arg (x))\sin (\arg(y)-\arg (x)),\\
|y-x|&\leq \frac{2}{\pi}|\arg (y) - \arg (x)|
\end{align*}
for $x,y \in \mathcal{B}_R$. Thus by \lref{F>C}, there exists $C>0$ such that
\begin{gather}
f(X_t)-f(X_0)\geq \int_0^tdf(X_{s-})\, \eta dX_s +C[\iota (X),\iota(X)]_t.\label{submartingale}
\end{gather}
In particular by \eqref{submartingale}, $\int_0^tdf(X_{s-})\, \eta dX_s$ is a submartingale bounded from above. Thus by taking expectations in \eqref{submartingale}, we obtain
\[
\mathbb{E}\left[ [\iota(X),\iota(X)]_{\infty} \right] \leq \frac{2}{C}.
\]
Thus by \tref{conv2}, $X_t$ converges as $t\to \infty$ almost surely.
\end{proof}
\begin{cor}
Suppose $Z$ is conservative. Let $u\colon E \to \mathcal{B}_R$ be a quasi-harmonic map on $E$ with respect to $Z$. We further suppose that any bounded harmonic function on $E$ with respect to $Z$ is constant. Then $u$ is constant quasi-everywhere on $E$.
\end{cor}
\begin{proof}
Let
\[
Y=\liminf _{t\to \infty} \iota (\tilde u(Z_t)).
\]
Then
\[
Y\circ \theta_s=Y,\ s\geq 0.
\]
On the other hand, there exists an $m$-polar set $N$ such that $\displaystyle \lim_{t\to \infty} \tilde u(Z_t)$ exists in $M$ and $\displaystyle \lim_{t\to \infty} \iota(\tilde u(Z_t))=Y$ $\mathbb{P}_z$-a.s. for all $z\in E\, \backslash \, N$ by \pref{conv3}. Since every bounded harmonic function is constant on $E$ by assumption, $\displaystyle \lim_{t\to \infty} \iota (\tilde u(Z_t))$ is a non-random point $y\in \iota (M)$ under $\mathbb{P}_z$ for $z\in E\, \backslash \, N$ and the limit does not depend on $z$. Let $x\in M$ be a point such that $\iota (x)=y$. Then
\[
\lim_{t\to \infty}\tilde u(Z_t)=x,\ \mathbb{P}_z\text{-a.s}.
\]
On the other hand, by \cite[Propositions 3.4 and 6.1]{Pic2}, there exists a distance function $\delta$ on $\mathcal{B}_{\frac{\varepsilon}{2}}$ such that $\delta$ is equivalent to the standard Riemannian distance on $\mathcal{B}_R$. Moreover, if we let
\[
\psi_x(y):=\delta (x,y)^p,\ y\in \mathcal{B}_R,
\] 
for fixed $p>\frac{1}{\sin (\pi -R)}$, then $\psi_x(\tilde u(Z))$ is a non-negative bounded submartingale. Hence $\psi_x(\tilde u(Z))$ converges $\mathbb{P}_z$-a.s. for q.e. $z$ and
\[
\lim_{t\to \infty}\psi_x(\tilde u(Z_t))=\psi_x(x)=0.
\]
Since $\psi_x$ is non-negative, for all $s\leq t$ and $\Lambda \in \mathcal{F}_s$, it holds that
\[
0\leq \mathbb{E}\left[ \psi_x(\tilde u(Z_s))\mathbf{1}_{\Lambda} \right] \leq \mathbb{E}\left[ \psi_x(\tilde u(Z_t))\mathbf{1}_{\Lambda} \right].
\]
Applying the bounded convergence theorem, we obtain $\mathbb{E}\left[ \psi_x(\tilde u(Z_s))\mathbf{1}_{\Lambda} \right] =0$. Since this holds for any $\Lambda \in \mathcal{F}_s$ and $s\geq 0$, we have $\psi_x(\tilde u(Z_s))=0$, $\mathbb{P}_z$-a.s. for q.e. $z$. Therefore $\tilde u(Z_t)=x$, $t \geq 0$. In particular  $\tilde u(z)=x$ for q.e. $z$.
\end{proof}
Next we consider singularities of harmonic maps. We suppose the transition function of a Markov process $Z$ satisfies the \textbf{condition of absolute continuity}, that is, there exists a Borel measurable function $p\colon [0,\infty)\times E\times E\to [0,\infty]$ such that
\[
p_tf(z)=\int_E f(w)p(t,z,w)\, m(dw),\ \text{for all}\, f\in \mathcal{B}_+(E).
\]
For a quasi-harmonic map $u$ with respect to $Z$, whether $\displaystyle \lim_{t\to 0}u(Z_t)$ exists or not under $\mathbb{P}_z$ is related to singularities of $u$. Then $u$ is said to be finely continuous at $z\in E$ if it holds that
\[
\mathbb{P}_z \left( \lim_{t\to 0} u(Z_t)=u(Z_0) \right) =1.
\]
We show some lemmas before the application of \tref{convtto0}. We remark that a semimartingale $X$ is said to be \textbf{special} if it can be decomposed as
\begin{gather}\label{canonical}
X=X_0+L+A,\ L_0=A_0=0,
\end{gather}
where $L$ is a local martingale and $A$ is a predictable process of locally integrable variation. We call the decomposition \eqref{canonical} the \textbf{canonical decomposition}.
\begin{lem}\label{martingaledecomposition}
Let $M$ be a compact Riemannian submanifold in $\mathbb{R}^d$ and $(\{X_t\}_{t\geq 0},\zeta, p)$ an $\eta$-martingale with an end point $p$ on a probability space $(\Omega,\mathcal{F},\{\mathcal{F}_t\}_{t\geq 0},\mathbb{P})$. Denote the canonical decomposition of $X$ by
\[
X=X_0+L+A.
\]
Let $\iota \colon M\to \mathbb{R}^d$ be an inclusion map and $\bar{\iota}$ an extension of $\iota$ satisfying \eqref{nabla}. Then
\[
\int_0^t \langle D \bar{\iota^i}(X_{s-})\, ,dA_s\rangle=0
\]
almost surely for all $t\geq 0$ and $i=1,\dots,d$.
\end{lem}
\begin{proof}
Since $X$ is an $\eta$-martingale with an end point, $\int_0^t \langle D \bar{\iota^i}(X_{s-})\, ,d\bar{\iota} (X_s)\rangle$ is a local  martingale. Now we have
\begin{align}\label{intA}
\int_0^t \langle D \bar{\iota^i}(X_{s-})\, ,dA_s\rangle&= \int_0^t \langle D \bar{\iota^i}(X_{s-})\, ,d\bar{\iota} (X_s)\rangle - \int_0^t \langle D \bar{\iota^i}(X_{s-})\, ,dL_s\rangle.
\end{align}
Since $D \bar{\iota}$ is bounded, \eqref{intA} implies that $\int_0^t \langle \nabla \bar{\iota^i}(X_{s-})\, ,dA_s\rangle$ is a predictable local martingale of locally bounded variation. Thus it equals zero.
\end{proof}
Next we will show that if $u$ is a quasi-harmonic map on an open set $D\subset E$, $\{u(Z_{t\land \tau_{D_1}}) \}_{t>0}$ is a $\mathbb{P}_z$-$\eta$-martingale with the end point $0$ for any $z\in D$ and relatively compact open subset $D_1$ satisfying
\begin{gather}
z\in D_1,\ \overline{D_1}\subset D \label{D1}
\end{gather}
under the absolute continuity of $\{p_t \}_{t\geq 0}$. Such a fact has already been mentioned in \cite{Pic4} in continuous cases, but we will give a proof in our situation including discontinuous cases.
\begin{lem}\label{absoluteconti}
Suppose the transition function of $Z$ satisfies the condition of absolute continuity. Let $M$ be a compact Riemannian submanifold in $\mathbb{R}^d$, $D\subset E$ an open set and $u\colon E\to M$ a quasi-harmonic map on $D$. Then $\{u(Z_{t\land \tau_{D_1}}) \}_{t>0}$ is a $(\mathbb{P}_z,\{\mathcal{F}^Z_t \}_{t>0})$-$\eta$-martingale with the end point $0$ for any $z\in D$ and relatively compact open subset $D_1$ satisfying \eqref{D1}.
\end{lem}
\begin{proof}
Let $(\Omega, \{Z^{D_1}_t\}_{t\geq 0}, \zeta^{D_1}, \{\mathbb{P}_z\}_{z\in D_1})$ be a part process of $Z$ on $D_1$. Then it holds that
\[
\tilde u(Z^{D_1}_t)-\tilde u(Z^{D_1}_0)=\tilde u(Z_{t\land \tau_{D_1}})-\tilde u(Z_0)-\tilde u(Z_{\tau_{D_1}})\mathbf{1}_{\{t\geq \tau_{D_1}\}}.
\]
Since $u$ is a quasi-harmonic on $D$, there exists a properly exceptional set $N$ such that $\{\tilde u(Z^{D_1}_t)\}_{t\geq 0}$ is an $\mathbb{R}^d$-valued $(\mathbb{P}_z,\{\mathcal{F}^Z_t\}_{t\geq 0})$-semimartingale for all $z\in D_1\, \backslash \, N$. Thus for each $i=1,\dots,d$, there exists a local MAF $L^{D_1,i}$ and an AF $A^{D_1,i}$ of locally integrable variation such that $L^{D_1}+A^{D_1}=(L^{D_1,1}+A^{D_1,1},\dots,L^{D_1,d}+A^{D_1,d})$ is a canonical decomposition of the special semimartingale $\tilde u(Z^{D_1})-\tilde u(Z_0)$ under $\mathbb{P}_z$ for all $z\in D_1\, \backslash \, N$ by Theorem (3.18) in \cite{CJPS}. We set
\begin{align*}
B_t&:=\tilde u(Z_{\tau_{D_1}})\mathbf{1}_{\{t\geq \tau_{D_1}\}},\\
\tilde B_t&:=\int_0^{t\land \tau_{D_1}}\int_E\left((\mathbf{1}_{E_{\Delta}\, \backslash \, D_1}\tilde u)(Z_s)-(\mathbf{1}_{E_{\Delta}\, \backslash \, D_1}\tilde u)(w)\right) \, N(Z_s,dw)dH_s,
\end{align*}
where $(N,H)$ is a L\'evy system of $Z$, which is a pair of a kernel $N(z, dw)$ and a PCAF $H$ of $Z$ satisfying
\[
\mathbb{E}_z \left[ \sum_{0<s\leq t}Y_sf(Z_{s-},Z_s) \right] =\mathbb{E}_z\left[ \int_0^{\infty}Y_s\left( \int_{E_{\Delta}}f(Z_s,w)\, N(Z_s,dw) \right) \, dH_s \right]
\] 
for any non-negative predictable process $\{Y_s\}$ and any $f\in \mathcal{B}_{+}(E_{\Delta} \times E_{\Delta})$ with $f(w,w)=0$, $w\in E$. Then $\tilde B$ is a CAF of $Z^{D_1}$ and it is the $\mathbb{P}_z$-dual predictable prediction of $B$. We further set
\begin{align*}
L_t&:=L^{D_1}_t+B_t-\tilde B_t,\\
A_t&:=\tilde B_t.
\end{align*}
Then $\tilde u(Z_{t\land \tau_{D_1}})-\tilde u(Z_0)=L_t+A_t$ is the canonical decomposition. Thus by \lref{martingaledecomposition}, it holds that
\begin{gather}\label{intA2}
\int \langle \nabla \bar{\iota^i}(u(Z_{-})^{\tau_{D_1}})\, ,dA\rangle=0.
\end{gather}
Since $B$ is bounded, $B-\tilde B$ is a $\mathbb{P}_z$-square integrable martingale for q.e. $z\in D_1\, \backslash \, N$ by the lemma on p. 160 in \cite{Pro}. Moreover, there exists an increasing sequence of nearly Borel finely open subsets $\{ G_n \}_{n=1}^{\infty}$ such that $\bigcup_{n=1}^{\infty}G_n=D_1\, \backslash \, N$, $\displaystyle \mathbb{P}_z(\lim_{n\to \infty}\sigma_{D_1\, \backslash \, G_n}=\infty)=1$ for all $z\in D_1\, \backslash \, N$, and
\begin{gather}\label{Lemma4.7}
\sup_{z\in D_1\, \backslash \,N}\mathbb{E}_z\left[\int_{(0,\sigma_{D_1\, \backslash \, G_n}]}e^{-s}\, d[L^{D_1},L^{D_1}]_s\right]<\infty
\end{gather}
by Lemma (4.7) in \cite{CJPS}. In particular, $L^{D_1, \sigma_{D_1\, \backslash \, G_n}}$ is a $\mathbb{P}_z$-square integrable for all $n\in \mathbb{N}$ and $z\in D_1\, \backslash \, N$ since for fixed $t\geq 0$, it holds that
\begin{align*}
\sup_{z\in D_1\, \backslash \,N}\mathbb{E}_z\left[ [L^{D_1, \sigma_{D_1\, \backslash \, G_n}},\right.&\left.  L^{D_1, \sigma_{D_1\, \backslash \, G_n}}]_t \right] \\
&\leq e^t e^{-t}\sup_{z\in D_1\, \backslash \,N}\mathbb{E}_z\left[ [L^{D_1, \sigma_{D_1\, \backslash \, G_n}},L^{D_1, \sigma_{D_1\, \backslash \, G_n}}]_t\right] \\
&\leq  e^t \sup_{z\in D_1\, \backslash \,N}\mathbb{E}_z\left[ \int_{(0,\sigma_{D_1\, \backslash \, G_n}\land t]}e^{-s}\, d[L^{D_1},L^{D_1}]_s\right] <\infty
\end{align*}
by \eqref{Lemma4.7}. Thus $L^{\sigma_{D_1\, \backslash \, G_n}}$ is $\mathbb{P}_z$-square integrable martingale for each $n\in \mathbb{N}$ and $z\in D_1\backslash \, N$. To simplify the notation, we set $\tau=\tau_{D_1}$, $\sigma=\sigma_{D_1\, \backslash \, G_n}$ for fixed $n\in \mathbb{N}$. Fix $\varepsilon>0$. Let $\Lambda$ be an intersection of defining sets of $L$ and $A$.  Since $N$ is a polar set under the condition of absolute continuity, we have
\[
\mathbb{P}_z(\theta_{\varepsilon}^{-1}\Lambda)=\mathbb{E}_z\left[\mathbb{E}_{Z_{\varepsilon}} \left[ \mathbf{1}_{\Lambda} \right] \right]=1
\]
for all $z\in E$. On $\theta_{\varepsilon}^{-1}\Lambda$, it holds that
\begin{align*}
\tilde u(Z_{(t+\varepsilon)\land \tau \land \sigma})-&\tilde u(Z_{\varepsilon \land \tau \land \sigma})=\left( L_{t\land \sigma}\circ \theta_{\varepsilon}+A_{t\land \sigma}\circ \theta_{\varepsilon}\right)\mathbf{1}_{\{\varepsilon < \tau \land \sigma\}}
\end{align*}
for each $n$. We will show that $t\mapsto L_{t\land \sigma}\circ \theta_{\varepsilon}\mathbf{1}_{\{\varepsilon < \tau \land \sigma \}}$ is a martingale. We take $s, t \geq 0$ and set
\begin{align*}
\Omega_1&:=\{s<(\tau \land \sigma) \circ \theta_{\varepsilon},\ \varepsilon < \tau \land \sigma \}\cup \{\sigma \circ \theta_{\ep} = \infty,\ \varepsilon < \tau \land \sigma\},\\
\Omega_2&:=\{s\geq (\tau \land \sigma) \circ \theta_{\varepsilon},\ \sigma \circ \theta_{\ep}<\infty,\ \varepsilon < \tau \land \sigma \}.
\end{align*}
Then
\[
\Omega_1 \cup \Omega_2=\{ \ep <\tau \land \sigma \}.
\]
For $\omega \in \theta_{\varepsilon}^{-1}\Lambda$, it holds that
\begin{align*}
L_{(t+s)\land \sigma} \circ \theta_{\varepsilon}(\omega)\mathbf{1}_{\{\varepsilon < \tau \land \sigma \}} (\omega)
&= L_{(t+s)\land \sigma (\theta_{\varepsilon}\omega)}(\theta_{\varepsilon}\omega)\mathbf{1}_{\{\varepsilon < \tau \land \sigma \}}(\omega)\\
&= L_{(t+s)\land \sigma (\theta_{\varepsilon}\omega)}(\theta_{\varepsilon}\omega)\left(\mathbf{1}_{\Omega_1}+\mathbf{1}_{\Omega_2} \right)(\omega).
\end{align*}
By a simple calculation, we have
\begin{align*}
L^{D_1}_{(s+t)\land \sigma}(\omega)&=L^{D_1}_{s\land \sigma}(\omega)+L^{D_1}_{t\land \sigma}\circ \theta_s (\omega),\\
B_{(s+t)\land \sigma}&=B_{s\land \sigma}(\omega)+B_{t\land \sigma}\circ \theta_s (\omega),\\
\tilde B_{(s+t)\land \sigma}&=\tilde B_{s\land \sigma}(\omega)+\tilde B_{t\land \sigma}\circ \theta_s (\omega)
\end{align*}
for $\omega \in \Lambda$ with $s<(\tau \land \sigma) (\omega)$ or $\sigma (\omega)=\infty$. Therefore it holds that
\begin{align*}
L_{t\land \sigma (\theta_{\varepsilon} \omega)}(\theta_{\varepsilon}\omega)=L_{s\land \sigma (\theta_{\ep}(\omega))}(\theta_{\varepsilon} \omega) + L_{t\land \sigma (\theta_{s+\varepsilon} \omega)}(\theta_{s+\varepsilon}\omega)
\end{align*}
for $\omega \in \Omega_1 \cap \theta_{\varepsilon}^{-1}(\Lambda)$. Thus it holds that
\begin{align*}
&L_{(t+s)\land \sigma} \circ \theta_{\varepsilon}(\omega)\mathbf{1}_{\{\varepsilon < \sigma\}} (\omega)\\
&= \left( L_{s\land \sigma }\circ \theta_{\varepsilon}(\omega) + L_{t\land \sigma}\circ \theta_{s+\varepsilon}(\omega) \right)\mathbf{1}_{\Omega_1}(\omega) +L_{\sigma }\circ \theta_{\varepsilon} (\omega) \mathbf{1}_{\Omega_2}(\omega).
\end{align*}
By the last equality, we obtain
\begin{align*}
&\mathbb{E}_z \left[ L_{(t+s)\land \sigma} \circ \theta_{\varepsilon}\mathbf{1}_{\{\varepsilon < \tau \land \sigma \}} \mid \mathcal{F}_{s+\varepsilon} \right]\\
&= \mathbb{E}_z \left[ \left( L_{s\land \sigma} \circ \theta_{\varepsilon}+L_{t\land \sigma} \circ \theta_{s+\varepsilon}\right)\mathbf{1}_{\Omega_1}+L_{\sigma}\circ \theta_{\varepsilon}\mathbf{1}_{\Omega_2}\mid \mathcal{F}_{s+\varepsilon} \right]\\
&= L_{s\land \sigma}\circ \theta_{\ep} \mathbf{1}_{\Omega_1}+\mathbb{E}_{Z_{s+\ep}} \left[ L_{t\land \sigma} \right] \mathbf{1}_{\{\ep < \tau \land \sigma\}}+L_{\sigma}\circ \theta_{\varepsilon}\mathbf{1}_{\Omega_2}\\
&= L_{s\land \sigma}\circ \theta_{\ep}\mathbf{1}_{\{\ep < \tau \land \sigma\}}.
\end{align*}
Therefore $\{L_{t\land \sigma}\circ \theta_{\ep}\mathbf{1}_{\{\varepsilon < \tau \land \sigma\}} \}_{t\geq 0}$ is a $(\mathbb{P}_z,\{ \mathcal{F}_{t+\varepsilon} \}_{t\geq 0})$-martingale for all $z\in D_1$ and $\ep >0$. On the other hand, for $\omega \in \theta^{-1}_{\ep}\Lambda$, $t\mapsto A_{t\land \sigma (\theta_{\ep}\omega)}(\theta_{\ep}\omega)$ is a continuous function of locally bounded variation and it holds that
\begin{align*}
\int \langle D \bar{\iota^i}(&\tilde u(Z_{(\varepsilon + s)\land \tau (\omega)\land \sigma (\omega)-}(\omega)))\, ,d\left(A_{s\land \sigma} \circ \theta_{\varepsilon}(\omega) \mathbf{1}_{\{\ep < \tau \land \sigma \}} (\omega)\right)\rangle \\
&= \int \langle D \bar{\iota^i}(\tilde u(Z_{s\land \tau \land \sigma -})\circ \theta_{\varepsilon}(\omega)\, ,d\left(A_{s\land \sigma} \circ \theta_{\varepsilon}(\omega) \right)\rangle \mathbf{1}_{\{\ep < \sigma \}} (\omega)\\
&=0
\end{align*}
by \eqref{intA2}. Thus $\{ \tilde u(Z_{t+\ep})^{\tau_{D_1}\land \sigma_{D_1\, \backslash \, G_n}}\}_{t\geq 0}$ is a $\mathbb{P}_z$-$\eta$-martingale with the end point $0$ for all $\ep>0$, $z\in D_1$ and $n\in \mathbb{N}$. In particular, $\{\tilde u(Z_{t\land \tau_{D_1}}) \}_{t>0}$ is a $(\mathbb{P}_z,\{\mathcal{F}_t\}_{t>0})$-$\eta$-martingale with the end point $0$ for all $z\in D$.
\end{proof}
\begin{prop}
Suppose the transition function of $Z$ satisfies the condition of absolute continuity. Let $M$ be a compact Riemannian submanifold in $\mathbb{R}^d$ and $u\colon E\to M$ a quasi-harmonic map on an open set $D\subset E$. Then $u$ can be modified to be finely continuous at a point $z\in D$ if and only if there exists a relatively compact open subset $D_1$ satisfying \eqref{D1} such that it holds that
\begin{gather}
\lim_{\varepsilon \to 0}[\tilde u(Z)^{\tau_{D_1}},\tilde u(Z)^{\tau_{D_1}}]^{\varepsilon}_t<\infty \ \mathbb{P}_z\text{-a.s.}\label{finelyconti}
\end{gather}
\end{prop}
\begin{proof}
Under the assumption, $\{\tilde u(Z_{t\land \tau_{D_1}})\}_{t>0}$ is a $\mathbb{P}_w$-$\eta$-martingale on $(0,\infty)$ for all $w\in D$ and $D_1$ by \lref{absoluteconti}. Suppose $\tilde u$ is finely continuous at a point $z\in D$. Then $\displaystyle \lim_{t\to 0}\tilde u(Z_{t\land \tau_{D_1}})$ exists $\mathbb{P}_z$-almost surely and equals $\tilde u(z)$. Thus by \tref{convtto0}, $\tilde u(Z)^{\tau_{D_1}}$ can be extended to a $\mathbb{P}_z$-$\eta$-martingale on $[0,\infty)$. In particular \eqref{finelyconti} holds. Next we suppose $\{\tilde u(Z_{t\land \tau_{D_1}})\}_{t>0}$ satisfies \eqref{finelyconti}. Then by \tref{convtto0} again, $\displaystyle \lim_{t\to 0}\tilde u(Z_{t\land \tau_{D_1}})$ exists $\mathbb{P}_z$-almost surely. Furthermore the limit is constant under $\mathbb{P}_z$ by Blumenthal's 0-1 law. Thus $u$ has a modification which is finely continuous at $z$.
\end{proof}
\section*{Acknowledgements}
The author is grateful to Professor Hariya, his supervisor, for his careful reading of the previous version of the manuscript and for helpful comments. The author also thanks a referee for his/her constructive comments and corrections which have led to significant improvement of this paper. 

\end{document}